\documentclass[final]{siamltex}
 
\usepackage{amsmath,graphics,epsfig,rotating,amssymb,amsbsy,psfrag}
\usepackage[latin1]{inputenc}

\newcommand \be {\begin{equation}} 
\newcommand \ee {\end{equation}} 
\newcommand{\RR}{\mathbb{R}}

\newcommand{\Bvec}{\mathbf{B}}

\newcommand{\gvec}{\mathbf{g}}
\newcommand{\hvec}{\mathbf{h}}
\newcommand{\xvec}{\mathbf{x}}
\newcommand{\phivec}{\boldsymbol{\phi}}
\newcommand{\uvec}{\mathbf{u}}
\newcommand{\vvec}{\mathbf{v}}
\newcommand{\normal}{\mathbf{n}}

\newcommand{\Div}{\nabla \cdot}
\newcommand{\Grad}{\nabla}

\newcommand{\Lp}{L^p(\Omega)}
\newcommand{\Yp}{Y^p(\partial \Omega)}
\newcommand{\Wcurl}{W^{p}(\text{curl};\Omega)}
\newcommand{\Wimp}{W^{p}_{\text{imp}}(\text{curl};\Omega)}
\newcommand{\Wdiv}{W^{p}(\text{div};\Omega)}
\newcommand{\Wcurli}{W^{p}(\text{\emph{curl}};\Omega)}
\newcommand{\Wdivi}{W^{p}(\text{\emph{div}};\Omega)}
\newcommand{\Sobolev}[2]{W^{{#1},{#2}}(\Omega)}
\newcommand{\Besov}[3]{B^{{#1}}_{{#2},{#3}}(\partial  \Omega)}

\newcommand{\curl}[1]{\nabla \times #1} 

\title{The $p$-CurlCurl : Spaces, traces, coercivity and a Helmholtz decomposition in $L^p$}

\author{ 
     M{\scriptsize arc} L{\scriptsize aforest}\thanks{\'Ecole Polytechnique de Montr\'eal,
     D\'epartement de math\'ematiques et g\'enie industriel, C.P. 6079 succ. Centre-ville, 
	Montr\'eal, Qu\'ebec, Canada, H3C 3A7 ({\tt Marc.Laforest@polymtl.ca}).}
 }

\begin{document}

\maketitle


\begin{abstract}
This work provides the foundation for the finite element analysis of an elliptic problem which is the rotational analogue of the $p$-Laplacian and
which appears as a model of the magnetic induction in a high-temperature superconductor operating near it's critical current.
Whereas the function theory for the $p$-Laplacian requires standard results in $L^p$ Sobolev spaces, this problem requires an extension to $L^p$ spaces of
the well-known $L^2$ theory for divergence free vector fields, as used in the finite element method applied to incompressible flows and electromagnetic radiation.
Among other things, the analysis  requires extensions to $L^p$ of the well-known 
$H(\operatorname{div}; \Omega)$ and $H(\operatorname{curl};\Omega)$, extensions of traces and Green's theorem, 
a Helmholtz decomposition and finally a Friedrich's inequality. In this paper, we provide a proof of the existence and uniqueness
of weak solutions of our so-called $p$-CurlCurl problem. In a companion paper,
the analysis is extended to treat continuous and finite element solutions of the 
nonlinear parabolic problem whose spatial term is the $p$-CurlCurl operator.
Many of the results presented here are either already known, known in slightly different forms or are proven with the help of techniques 
that are already well-known. 
The main novelty of this paper appears to be the structured form of this $L^p$ theory and
our form of the Helmholtz decomposition and of the  Friedrich's inequality. 
In this respect, we note that some of these results can be found in the works of M. Dauge, M. Mitrea, I. Mitrea and ...
\end{abstract}



\pagestyle{myheadings}
\thispagestyle{plain}
\markboth{M. Laforest}{The $p$-CurlCurl : Function theory}


\section{Introduction}
\label{sec:intro}

The objective of this paper is to study the existence and uniqueness of weak
solutions of a nonlinear elliptic problem from applied superconductivity \cite{BraGriMar07,Sirois}. With respect to the magnetic flux 
$\Bvec: \Omega \to \RR^3,$ defined over a bounded domain $\Omega$, the stationnary problem takes the form
\begin{align}  
     \curl \Big( \big| \curl \Bvec \big|^{p-2}  \curl \Bvec  \Big) & = \mathbf{S}, \text{ over $\Omega$,}  \label{p-curl}  \\
      \Div \Bvec & = 0, \text{ over $\Omega$,}                                                                                                          \label{divergence_free}\\  
      \normal \times \Bvec & = 0, \text{ over $ \partial \Omega$,}                                             \label{Dirichlet_BC} 
\end{align}
where $p\in (1,\infty)$, $\normal$ is the outwards normal along the boundary $\partial \Omega$, $\mathbf{S}$ is given and satisfies some additional
conditions, and the boundary $\partial \Omega $ is sufficiently smooth. 
It is quite obvious that this problem is closely related to the $p$-Laplacian whose highest order term
is of the form $\Div ( | \Grad \mathbf{u} |^{p-2} \Grad \mathbf{u} ) $ and for which a well-developed theory already exists 
\cite{GloMar75,Wei92,BarLiu93,DiB93}.  For this reason, we have taken the liberty of naming this problem the elliptic $p$-CurlCurl.

In this paper, we tackle problem \eqref{p-curl}-\eqref{Dirichlet_BC} by constructing extensions of standard variational techniques from 
electromagnetics \cite{Mon03}. Roughly speaking, these extensions take the form of generalizations of results from $L^2$ to $L^p$
where $p \in (1,\infty) $ but in practice very large. Although many of the results 
described in this paper are scattered throughout the literature,  we have found it useful 
to write them in a self-contained form for future reference. For the engineering community which might not have access to the
mathematics literature, we have endeavoured to provide detailed
references and attributions to the results presented, while still providing a complete account of those results.
One reason why such a theory did not already exist is probably that a problem requiring it, like the $p$-CurlCurl, had yet to come
to the attention of the applied mathematics community.

The main contributions of this paper, besides the existence and uniqueness result for the elliptic $p$-CurlCurl problem, are the forms
of the Helmholtz decomposition, Theorem \ref{thm:Helmholtz}, and Friedrich's inequality, Theorem \ref{Friedrich_inequality2}. 
Ne\v{c}as \cite{Nec67}, Grisvard \cite{Gri85} and Taylor \cite{Tay81} have already given detailed presentations 
of elliptic regularity results in $L^p$, in the spirit of the groundbreaking work for smooth domains 
of Agmon, Douglis and Nirenberg \cite{AgmDouNir59}. The earliest
references to the $L^p$ analogues of $H(\text{div})$ and $H(\text{curl})$ we have found were those in Mitrea \cite{MitMit02,Mit04,Mit04}
but  they have also appeared implicitly in the work of Dauge \cite{Dau88,Dau92}. For smooth domains, the Helmholtz decomposition
in $W^{s,p}$ can be found in \cite{CanMat85} while for $C^1$ domains and Besov spaces, the question has recently 
been studied in \cite{FujYam07}.
Our Helmholtz decomposition is slightly different then those in either reference, mostly because the decomposition is not done over
$L^p$ but over the $L^p$ analogue of $H(\text{div})$. We were unable to find a published account 
of Friedrich's inequality in $L^p$ even though in theory, it very closely related to the Helmholtz decomposition.

The theory developed in this paper is presented under the strong assumption that the domain $\Omega$ be bounded with a $C^1$ boundary.
The boundedness of the domain is only present to simplify the proofs while the constraint on the boundary is impossible to overcome for the
large values of $p$, between $5$ and $100$ \cite{Sirois}, which one typically encounters in the engineering problems underlying the $p$-CurlCurl.
This strong assumption requires some explanation since it excludes the type of domain which one would typically encounter in finite element
discretizations. Going back to  counter-examples of Dahlberg \cite{Dah79} and the work of Jerison and Kenig on the Poisson
problem \cite{JerKen81,JerKen95}, it is known that  singularities in the smoothness of the boundary impose restrictions on the existence theory in $L^p$ spaces. 
This work has very recently been extended to general mixed boundary conditions for elliptic boundary value problems 
by Mitrea, ... \cite{Mit} where they show there exists
a neighborhood of $(2,2)$ for the values $(s,p)$, where $s$ the regularity being $s$ and $p$ the summability, for which the problem is 
well-posed; see Section \ref{sec:spaces}. Mention that Mitrea uses different techniques.
As such, the constraint imposed on the smoothness on the boundary is a fundamental obstacle which will require the engineering community to 
improve their model for the resistivity. Interesting work in that direction is being done by \cite{Sirois}.

For those interested in the modeling assumptions underlying the use of the $p$-CurlCurl in applied superconductivity, we refer the
reader to either Section 2 of the companion paper \cite{Laf10a} or to ...
Mention analytic solutions available in the literature for the time-dependent problem.

This paper is organized along the lines of Chapter 3 and 4 of Monk's recent monograph \cite{Mon03} treating finite element methods in electromagnetics,
which itself borrowed heavily from \cite{AmrBerDauGir98} and \cite{GirRav86}.  Section \ref{sec:regularity} reviews 
definitions of Sobolev and Besov spaces and 
presents precise statements concerning a few basic elliptic problems with boundary conditions in Besov spaces. Section \ref{sec:spaces}
introduces the $L^p$ analogues of $H(\text{div};\Omega)$ and $H(\text{curl};\Omega)$ and proves Green's formulas for them. The fourth section
construct scalar and vector potentials in $W^{1,p}$ for vector fields that are respectively curl and divergence free. The fifth section deals with 
Friedrich's inequality and the Helmholtz decomposition. Section \ref{sec:applications} proves
that there exists a weak solution to the $p$-CurlCurl problem.

\section{Elliptic regularity results}
\label{sec:regularity}

This section presents a summary of some necessary existence results for elliptic boundary value problems, in particular 
for the Poisson problem. We also include a description of Sobolev and Besov spaces and a list of their main properties.
With the exception of Theorem \ref{thm:Crouzeix}, the results in this section are known and presented without proof. We refer the reader to
the original paper of Agmon et al. \cite{AgmDouNir59}, the monograph of Ne\v{c}as for $W^{k,p}$ with $k$ integral \cite{Nec67} and
Grisvard \cite{Gri85} or Taylor \cite{Tay81} for $W^{s,p}$ with $s$ real but greater than $2$. Although Jerison and Kenig \cite{JerKen81,JerKen95} 
have focused on Lipschitz domains, their results are presented in forms that are closer to those which we have presented below and therefore
represent a good reference.
The results are presented for domains  with $C^1$
boundaries and will later briefly comment on the nature of the restrictions on $p$ which one would encounter
if more general domains were considered. For people familiar with elliptic problems, this section can be safely skipped.

Consider an open and bounded domain $\Omega \subset \RR^n$ whose boundary $\partial \Omega$ possesses Lipschitz regularity; see
Definition \ref{defn:Lipschitz_boundary}. For $1 \leq p < \infty$, recall the definition of the $L^p$-norm
$$
    \| u \|_{L^p(\Omega)} := \left\{ \int_{\Omega} |u|^p \, dx \right\}^{1/p}
$$
and the Banach space of functions over $\Omega$,
$    L^p(\Omega) := \{ u \text{ measurable} \, | \,  \| u \|_p < \infty \} $. 
Using the $L^p$-norm to control regularity, for each multi-index  
 $\boldsymbol{\alpha} := (\alpha_1, \alpha_2, \ldots,  \alpha_n) \in \mathbb{N}^n$,  we define
$|\boldsymbol{\alpha}| = \alpha_1+\alpha_2 + \cdots + \alpha_n$ and a semi-norm
$$
    | u |_{\Sobolev{k}{p}} := \left\{   \sum_{\boldsymbol{\alpha} : |\boldsymbol{\alpha}| = k} 
           \left| \frac{ \partial^k u}{\partial^{\alpha_1} x_1 \cdots \partial^{\alpha_n} x_n } \right|_p^p  \right\}^{1/p} .
$$
We can thus construct a norm by using all semi-norms for derivatives  less than or equal to a fixed order, namely
$$
    \| u \|_{\Sobolev{k}{p}} := \left\{  \sum_{i=0}^k   | u |_{i,p}^p   \right\}^{1/p} ,
$$ 
and thus, by completion with respect to smooth functions over $\Omega$, the so-called Sobolev space
\begin{align}    \label{defn:Sobolev}
  \Sobolev{k}{p} = \Big\{  u \text{ measurable} \Big|  \| u \|_{k,p} < \infty  \Big\} .
\end{align}
The spaces \eqref{defn:Sobolev} are Banach spaces and when $p=2$, they are Hilbert spaces. 

To extend these spaces to all real positive real values of $s$, we define the norm 
$$
   \|  u \|_{\Sobolev{s}{p}} :=  \left\{ \| u \|_{\Sobolev{m}{p}}^p + 
         \sum_{|\boldsymbol{\alpha}| = m} \int_{\Omega} \int_{\Omega} 
             \frac{ \partial^{\boldsymbol{\alpha}} u(\xvec) - \partial^{\boldsymbol{\alpha}} u(\mathbf{y}) }{|\xvec - \mathbf{y}|^{n+\sigma p}}  
                                                         \, d\xvec d\mathbf{y} \right\}^{1/p} ,
$$
where $m$ is the positive integer and $\sigma\in [0,1)$ the real number satisfying $s = m+\sigma$. In \cite{Nec67}, it is shown
that the completition of smooth functions $C^{\infty}(\overline{\Omega})$ with respect to
the norm $\| \cdot \|_{\Sobolev{s}{p}}$, $1<p<\infty$, gives rise to a reflexive Banach space.


\begin{definition} \label{defn:Lipschitz_boundary}
We say that a bounded domain $\Omega$ is Lipschitz if at every point $\xvec \in \partial \Omega$, there exits a neighborhood $V \subset \RR^n$ 
of $\xvec$ such that $\partial \Omega \cap V$ can be described as the graph of a Lipschitz function. 
Similarly, we say that the domain is $C^m$ is it's boundary can be represented locally as the graph of
a  function which is $m$ times differentiable with continuous $m$-th order derivatives.
\end{definition}

It is a non-trivial result that for Lipschitz domains, $C^{\infty}(\overline{\Omega})$ is dense in $\Sobolev{k}{p}$ for $k \in \mathbb{N}$,
while this follows by construction when $s \in \mathbb{R}^+ \setminus \mathbb{N}$.
When $s$ is a non-negative real number and $1 \leq p < \infty$, it is well-known that when the boundary of the domain is Lipschitz, then the set of restrictions
to $\Omega$ of functions in $C^{\infty}_0(\RR^n)$ is dense in $\Sobolev{s}{p}$; Theorem 3.22 of \cite{Ada03}. This suggests that we define the spaces 
$W^{s,p}_0(\Omega)$ to be the closure of the set $C_0^{\infty}(\Omega)$ in the norm of $\Sobolev{s}{p}$. For any normed topological vector space $V$, we define
its dual $V'$ to be the set of continuous linear functionals $\ell : V \longrightarrow \RR$, which is also a normed topological vector space
with respect to the induced norm
\begin{equation}  \label{dual_norm}
                \| \ell \|_{V'} := \sup_{v \in V : \|v \|_V \leq 1} \big| \ell(v) \big| .
\end{equation}
When $1 < p < \infty$, then we define the conjugate exponent to $p$ to be $q$ satisfying $1 = 1/p + 1/q$. 
The Riesz representation theorem states that $L^p(\Omega)' = L^q(\Omega)$. When $s$ is  negative, we define
$\Sobolev{s}{p} \equiv W^{-s,q}_0(\Omega)'$ becomes a Banach spaces; see \cite{Ada03,Nec67} 
for more information on those spaces. 
Below, we will use the symbols $( u, v) $ to denote $\int u v \, dx$, where the 
subscript $\Omega$ will be included if the space over which the integration occurs is not obvious.  
The bracket notation  $  \langle \cdot, \cdot \rangle : V \times V' \longrightarrow \mathbb{R}$  will be reserved
for pairings between  $V$ and its dual $V'$.

One of the facts that makes handling Sobolev spaces difficult in the $L^p$ setting is that their traces do not belong to the same family of
spaces. This requires us to introduce the so-called Besov spaces $\Besov{s}{p}{p'}$ where $s$ is a measure of the regularity and $p$ and
$p'$ are two exponents in $[1, \infty)$ which might or might not be related.  
We will never need the explicit definition of these spaces, only the existence of these spaces and 
the existence of continuous maps to these spaces. The difficult definition of these spaces is therefore omitted and we refer
the reader to Chapter 7 of \cite{Ada03} or Treibel \cite{Treibel}. The fundamental properties of these fractional order spaces are the same as those of
Sobolev spaces, namely completeness and density 
of the subset of smooth functions. Moreover, if $B^{s}_{p,p;0}(\Omega)$ is the closure of $C^{\infty}_0(\Omega)$ in the
Besov space $B^s_{p,p}(\Omega)$, then the negative norm Besov spaces will be defined by duality as
$B^{s}_{p,p}(\Omega) := (B^{-s}_{q,q;0}(\Omega))'$, where $p$ and $q$ are conjugate to each other. 
We now include two fundamental results for these spaces.

\begin{lemma}[Lemma 7.41 of \cite{Ada03}]  \label{lemma:trace_basic}
Assume that $\Omega$ is a bounded Lipschitz domain and $1 < p < \infty$.
Then the restriction operator $\gamma_0 : C^{\infty}(\overline{\Omega}) \longrightarrow C^{\infty}(\partial \Omega)$ extends continuously
to a surjective map $\gamma_0: \Sobolev{1}{p} \longrightarrow \Besov{1-1/p}{p}{p}$.
\end{lemma}

\begin{lemma}[Poincar\'e's inequality, Remark 7.5 of \cite{Nec67}]  \label{lemma:Poincare}
Assume that $\Omega$ is a bounded Lipschitz domain and $1<p<\infty$.
Then there exists a constant $C = C(\Omega)$ such that 
$$
     \| u \|_{\Sobolev{1}{p}} \leq C \left\{ \| \nabla u \|_{\Lp}^p  + \left| \int_{\Omega} u \, d\xvec \right|^p \right\}^{1/p}  .
$$
\end{lemma}

We now present the fundamental well-posedness results we will be using. In the following theorems, we will be assuming that 
the domain $\Omega$ is bounded with a {\it smooth} $C^1$ boundary. 
Removing this constraint would require us to significantly reduce the 
range of the exponent which would be unacceptable, 
under the current modeling assumptions leading to the problem \eqref{p-curl}-\eqref{Dirichlet_BC}.
briefly addressing this issue at the end of the current section. The following results are parts of what is usually referred to as the shift theorem
for elliptic partial differential equations; see Theorem 5.4 \cite{Tay81}. In the form given below, these results are consequences of the 
work of Agmon, Douglis and Nirenberg \cite{AgmDouNir59} where they assumed throughout a smooth domain
and $L^p$-type boundary conditions. Below we give specific references, that are neither original nor optimal,
but have the benefit of being in precisely the same form as we have stated.

\begin{theorem} \label{thm:elliptic1}
Consider a bounded domain with a $C^{1}$ boundary, an exponent $1<p<\infty$. Then for any  $\mu \in \Besov{1-1/p}{p}{p}$ and
$f \in \Sobolev{-1}{p}$, there exists a unique weak solution $\phi \in \Sobolev{1}{p}$ of
$$
     - \Delta \phi + \phi = f, \quad \text{in $\Omega$} \qquad \text{and} \qquad \phi = \mu \quad \text{on $\partial \Omega$}.
$$
Moreover, there is a constant $C$ such that
\begin{equation}  \label{estimate1}
     \|  \phi \|_{\Sobolev{1}{p}} \leq C \Big( \| \mu \|_{\Besov{1-1/p}{p}{p}} + \| f \|_{\Sobolev{-1}{p}}  \Big) .
\end{equation}
\end{theorem}

\begin{theorem}[Theorem 6.1, \cite{JerKen95}] \label{thm:elliptic2}
Consider a bounded domain with a $C^{1}$ boundary, an exponent $1<p<\infty$. Then for any $\mu \in  \Besov{1-1/p}{p}{p}$ and
$f \in \Sobolev{-1}{p}$, there exists a unique weak solution $\phi \in \Sobolev{1}{p}$ of
$$
     - \Delta \phi  = f, \quad \text{in $\Omega$} \qquad \text{and} \qquad \phi = \mu \quad \text{on $\partial \Omega$}.
$$
Moreover, there is a constant $C$ such that
\begin{equation}  \label{estimate2}
     \|  \phi \|_{\Sobolev{1}{p}} \leq C \Big( \| \mu \|_{ \Besov{1-1/p}{p}{p}} + \| f \|_{\Sobolev{-1}{p}}  \Big) .
\end{equation}
\end{theorem}

\begin{theorem} \label{thm:elliptic4}
Consider a bounded domain with a $C^{1}$ boundary, unit outwards normal $\normal$, and an exponent $1<p<\infty$. 
Then for any $\mu \in \Besov{-1/p}{p}{p} $ 
and $f \in \Sobolev{1}{q}'$, there exists a unique weak solution $\phi \in \Sobolev{1}{p}$ of
$$
     - \Delta \phi + \phi = f, \quad \text{in $\Omega$} \qquad \text{and} \qquad \normal \cdot \nabla \phi = \mu \quad \text{on $\partial \Omega$}.
$$
Moreover, there is a constant $C$ such that
\begin{equation}  \label{estimate4}
     \|  \phi \|_{\Sobolev{1}{p}} \leq C \Big( \| \nu \|_{\Besov{-1/p}{p}{p}} + \| f \|_{\Sobolev{1}{q}'}  \Big) .
\end{equation}
\end{theorem}

\begin{theorem}[Theorem 9.2, \cite{FabMenMit98}] \label{thm:elliptic3}
Consider a bounded domain with a $C^{1}$ boundary,  unit outwards normal $\normal$, an exponent $1<p<\infty$,  
and any two functions $\mu \in \Besov{-1/p}{p}{p}$ and
$f \in \Sobolev{1}{q}'$ satisfying the compatibility condition
$$
             \int_{\Omega} f \, d\xvec + \int_{\partial \Omega} \mu \, d\boldsymbol{\sigma} = 0.
$$
Then there exists a weak solution $\phi \in \Sobolev{1}{p}$, unique up to an additive constant, of the problem
$$
     - \Delta \phi  = f, \quad \text{in $\Omega$} \qquad \text{and} \qquad \normal \cdot \nabla \phi = \mu \quad \text{on $\partial \Omega$}.
$$
Moreover, there is a constant $C$ such that
\begin{equation}  \label{estimate3}
     \|  \phi \|_{\Sobolev{1}{p}} \leq C \Big( \| \mu \|_{\Besov{-1/p}{p}{p}} + \| f \|_{\Sobolev{1}{q}'}  \Big) .
\end{equation}
\end{theorem}

We will also be needing the following additional regularity results.

\begin{theorem}[Theorem , \cite{Gri85}] \label{thm:elliptic_regularity}
Consider a bounded domain with a $C^{1}$ boundary,  outwards normal $\normal$, and an exponent $1<p<\infty$. 
For any  $\mu \in  \Besov{1}{p}{p}$, there exists a unique weak solution $\phi \in \Sobolev{1+1/p}{p}$ of
$  \Delta \phi  = 0$ in $\Omega$ and $ \phi = \mu $ on $\partial \Omega$. Moreover, there exists a constant $C$ for which
\begin{equation}  \label{estimate_regularity1}
           \| \phi \|_{\Sobolev{1+1/p}{p}} \leq C \, \| \mu \|_{W^{1,p}(\partial \Omega)} .
\end{equation}
On the other hand, for  any  $\mu \in  \Besov{1-1/p}{p}{p}$ satisfying 
$$ 
      \int_{\partial \Omega} \mu  \, d\boldsymbol{\sigma}= 0,  
$$
there exists a unique weak solution $\phi \in \Sobolev{1+1/p}{p}$ of
$  \Delta \phi  = 0$ in $\Omega$ and $ \normal \cdot \nabla \phi = \mu $ on $\partial \Omega$. Moreover,
for some constant $C$,
\begin{equation}  \label{estimate_regularity2}
           \| \phi \|_{\Sobolev{1+1/p}{p}} \leq C \, \| \mu \|_{\Besov{1-1/p}{p}{p}} .
\end{equation}
\end{theorem}

Include a discussion of the results of Mitrea.

Later in this paper, we will be needing three additional technical results.
The first of these  is demonstrated in $L^2$ in Section 2.2 of 
\cite{GirRav86} using two deep functional analytic results, the first from Peetre \cite{Pee66} and Tartar \cite{Tar78}, the second from Ne\v{c}as \cite{Nec66}.
Although we have not found the next lemma, as stated, within the literature, in fact both the result of Peetre and Tartar (stated for Banach spaces)
and the one of Ne\v{c}as (proved for $W^{k,p}$, $k$ integer and $1<p<\infty$) are sufficiently general that the proof of the next lemma, as
given in \cite{GirRav86}, applies verbatim. Since the result is rather tangential, technical and in any case, somewhat natural, we do not attempt a proof here.

\begin{lemma} \label{thm:Necas}
Assume that $\Omega \subset \RR^n$ has a Lipschitz boundary and $1<p <\infty$. 
If $\phi \in L^p_{\operatorname{loc}}(\Omega)$ and $\nabla \phi \in \Sobolev{-1}{p}$ then $\phi \in \Lp$.
\end{lemma}



The second technical result  is an extension of a density result of Ben Belgacem et al. \cite{BenBerCosDau97} based on a proof by Michel Crouzeix.

\begin{theorem}  \label{thm:Crouzeix}
Consider a bounded domain $\Omega \subset \RR^n$ with Lipschitz boundary and $1 < p < \infty$. 
Then $C^{\infty}(\overline{\Omega})$ is dense in the space
$$
     G = \left\{   u \in \Sobolev{1}{p} \big| \gamma_0(u) \in W^{1,p}(\partial \Omega) \right\}, 
$$
with norm $\| u\|_G = \| u \|_{\Sobolev{1}{p}}+ \| \gamma_0(u) \|_{W^{1,p}(\partial \Omega)}$.
\end{theorem}

\begin{proof}
The proof in \cite{BenBerCosDau97} applies mutatis mutandis. 

In a few words, the proof proceeds as follows. The boundary of the domain $\partial \Omega$
is covered by a finite set of open starlike sets $\{ \Omega_k \}_k$ each of which can therefore be parametrized by 
spherical coordinates $(r,\theta)$. The boundary is then described by $\partial \Omega_k = \{(r,\theta) | r = R_k(\theta)   \}$ 
where $R_k$ is a Lipschitz function. 
Choose $u \in G$ and, using the fact that $\gamma_0(u) \in W^{1,p}(\partial \Omega)$,
construct $u_k \in W^{1,p}(\Omega_k)$ such that $u_k  = u$ on $ \partial \Omega_k \cap \partial \Omega $ (for example,
take $u_k(r,\theta) = \alpha_k(r) u(R_k(\theta),\theta)$ where $\alpha_k \equiv 1$ but vanishes in a neighborhood of $r=0$).

With the help of a partition of unity for the covering $\{ \Omega_k \}$, a function $u_b \in W^{1,p}(\Omega)$ can be constructed
such that $\gamma_0(u_b) = \gamma_0(u)$. Then $u = u_b + u_0$ where $u_0 \in W^{1,p}_0(\Omega)$ can be approximated in 
$C^{\infty}_0(\Omega)$. Using the same construction as for $u_b$, we can approximate $\gamma_0(u)$ in $C^{\infty}(\partial \Omega)$
and extend the approximation into $\Omega$, thus providing an approximation of $u_b$.    
\end{proof}

The final ingredient is well-known in the literature as the Babu\v{s}ka-Lax-Milgram Theorem \cite{}. We provide a complete statement of the result because it hinges on two important conditions which we would need to define in any case.

\begin{theorem}[Babu\v{s}ka-Lax-Milgram]  \label{thm:Lax-Milgram}
Let $X$ and $Y$ be reflexive Banach spaces and consider
$B:X\times Y \longrightarrow \RR$ 
a continuous, bilinear form satisfying the following two conditions :
\begin{itemize}
\item[(i)] $B$ is non-degenrate with respect to the second variable, that is to say, for each non-zero $y \in Y$
                there exists $x \in X$ such that $B(x,y) \neq 0$ ;
\item[(ii)]  $B$ satisfies an $\inf$-$\sup$ condition, that is there exists a strictly positive constant $c$ such that
\begin{equation}     \label{inf-sup}
            c    \leq    \inf_{x \in X} \sup_{y \in Y}         \frac{\big| B(x, y) \big|}{\|x\|_{X} \| y \|_{Y} }.
\end{equation}              
\end{itemize}
Then for every $\ell \in Y'$, there exists a unique solution $x \in X$ to
$$
          B(x,y) = \ell(y), \qquad \forall y \in Y.
$$
Moreover, $c \| x \|_X \leq \| \ell \|_{Y'}$.
\end{theorem}

\section{Spaces and traces}
\label{sec:spaces}

In this section, we present some basic function theoretic results in function  spaces typically encountered in 
electromagnetism but usually in an $L^2$ setting \cite{Mon03}. 
The proofs of the results in this section can sometimes be found in the literature but the proofs
are often only briefly sketched and/or appear buried deep in the papers themselves.  In particular, we refer 
the interested reader to either the short paper \cite{MitMit02} or Section 2 of \cite{Mit04}, where the definitions are outlined for Riemannian manifolds,  
or  Section 9 of \cite{FabMenMit98} where the results are presented without the geometry but more succinctly.
At the risk of appearing pedantic, we have chosen to present these concepts in a 
logically complete fashion since this paper as a whole is likely to interest some engineers
with less experience with the mathematical literature.

 For electromagnetic problems, the divergence and the curl of a vector field must be well-defined. This implies that one must 
 understand the following two spaces
 \begin{align}   \label{defn:Wsp_div}
  \Wdiv & = \Big\{  \uvec \in  L^{p}(\Omega)^3\Big|  \Div \uvec  \in  L^p(\Omega) \Big\} , \\
 \Wcurl & = \Big\{  \uvec \in  L^{p}(\Omega)^3 \Big|  \curl \uvec \in   L^{p}(\Omega)^3 \Big\}  .
\end{align}
It is important to observe that these are again Banach spaces with norms, respectively
 \begin{equation}   \label{norm:Wsp_div}
  \begin{array}{rcl}
        \| \uvec \|_{\Wdiv}   &  := & \| \uvec \|_{\Lp} + \| \Div \uvec \|_{\Lp},  \\
        \| \uvec \|_{\Wcurl}  & := & \| \uvec \|_{\Lp} + \| \curl \uvec \|_{\Lp},
  \end{array}  
\end{equation}
and that there exist continuous injections $W^{1,p}(\Omega)^3 \hookrightarrow \Wdiv$ and
$W^{1,p}(\Omega)^3 \hookrightarrow \Wcurl$. On the other hand, these injections are not surjective.
We define the closure of $C^{\infty}_0(\Omega)^3$ in $\Wdiv $ and $\Wcurl $
to be respectively  $W^{p}_0( {\text{div}}; \Omega) $ and $W^{p}_0( {\text{curl}}; \Omega) $.

We need to know how to interpret the values along the boundary of functions  in either $ \Wdiv$ or $\Wcurl$.
We begin the analysis of $\Wdiv$ with a density result which will allow us to extends many operators quite naturally to
all elements in these spaces. Theorem \ref{thm:densityWdiv} is similar to Theorem 3.22 of \cite{Mon03} and its proof is simply an adaptation
to $L^p$ spaces.


\begin{theorem} \label{thm:densityWdiv}
Assuming that $\Omega$ has a Lipschitz boundary and $1<p<\infty$, then the set $C^{\infty}(\overline{\Omega})^3$ is dense in 
$W^{p}(\text{\emph{div}};\Omega)$.
\end{theorem}

The proof of this result requires three elementary lemmas.

\begin{lemma}
For $1<p<\infty$, the space $C^{\infty}_0(\RR^3)^3$ is dense in both $W^{p}(\text{\emph{div}};\RR^3)$ and $W^{p}(\text{\emph{curl}};\RR^3)$.
\end{lemma}

\begin{proof}
It suffices to observe that the norms in $W^p(\text{div};\RR^3)$ and $W^{p}(\text{curl};\RR^3)$
are bounded by the norm in $W^{1,p}(\RR^3)$. The result then follows by the density of $C^{\infty}_0(\RR^3)$
in $W^{1,p}(\RR^3)$, Corollary 3.23 of \cite{Ada03}.
\end{proof}

\begin{lemma} \label{lemma:Green_weak}
For $1<p<\infty$ and any $\uvec \in W^p(\text{\emph{div}};\RR^3)$, $\phi \in C^{\infty}_0(\RR^3) $, we have
\begin{equation} \label{Green_div_smooth}
      \int_{\RR^3} \Div \uvec \, \phi \, d\xvec +  \int_{\RR^3} \uvec \, \nabla \phi \, d\xvec = 0.
\end{equation}
For any $\uvec \in W^p(\text{\emph{curl}};\RR^3)$ and $\phivec \in C^{\infty}_0(\RR^3)^3 $, we have
\begin{equation}  \label{Green_curl_smooth}
      \int_{\RR^3} \curl \uvec \, \phivec \, d\xvec -  \int_{\RR^3} \uvec \,  \curl \phivec \, d\xvec = 0.
\end{equation}
\end{lemma}

\begin{proof} The result follows immediately by density using the classical Divergence Theorem.
\end{proof}


We now have all the tools to tackle the proof of the density of $C^{\infty}(\overline{\Omega})$ inside $W^p(\text{div};\Omega)$.
The main idea of the proof is that $C^{\infty}_0(\Omega)$ will always be dense in the dual of a "smooth" function space. The proof proceeds
by identifying this fact without having to obtain a complete characterization of the dual.

\begin{proof}[Proof of Theorem \ref{thm:densityWdiv}] The idea for this proof is taken from \cite{CosDau98}
where it was used to study a slightly different space. The techniques presented here will be reused 
in the proof of Theorem \ref{thm:densityWcurl}.
We will show that if $\ell \in \Wdiv'$ is identically zero over $C^{\infty}(\overline{\Omega})^3$
then $\ell \equiv 0$. By Theorem 3.5 of \cite{Rud91}, this suffices to show that the closure of $C^{\infty}(\overline{\Omega})^3$ equals $\Wdiv$.

Define $\Omega^{(4)}$ to be the union of four distinct copies of $\Omega$ and consider the obvious embedding
\begin{align*}
   I : \Wdiv & \longrightarrow L^p(\Omega^{(4)}) \\
        \uvec & \longmapsto (\uvec, \Div \uvec) .
\end{align*}
This embedding is an isometry onto it's image. Any linear functional $\ell$ over $\Wdiv$ defines
a linear functional over the closed image $I( \Wdiv ) $ and therefore, by the Hahn-Banach Theorem,
extends to a linear functional over all of $L^p(\Omega^{(4)})$. Since the dual of $L^p$ is $L^q$,  for 
any $\uvec \in \Wdiv $ this extension can be written in the form
$$
   \ell(\uvec) =  ( \uvec ,  \vvec )_{\Omega}  + ( \Div \uvec , w )_{\Omega}
$$
for some (non-unique) $\vvec \in L^q(\Omega)^3$ and $w \in L^q(\Omega)$. Let $E_0$ be the extension by
zero of any element in $L^q(\Omega)$ to $L^q(\RR^3)$ (Theorem 2.30 of \cite{Ada03}) and let $E$ be any extension operator
from $L^p(\Omega)$ to $L^p(\RR^3)$, then
$$
   \ell (\uvec) =  ( E \uvec , E_0 \vvec )_{\RR^3}  + ( \Div \uvec , w )_{\Omega} .
$$
Assuming that $\ell(\phivec) = 0$ for all $\phivec \in C^{\infty}(\overline{\Omega})^3$, then for all $\boldsymbol{\psi} \in C^{\infty}_0(\RR^3)^3$
we have
\begin{align*}
   ( \boldsymbol{\psi} , E_0 \vvec )_{\RR^3}  +  ( \Div \boldsymbol{\psi} , E_0 w )_{\RR^3}  = 0.
\end{align*}
In a distributional sense, we therefore have $E_0 \vvec = \nabla E_0 w$. This allows us to first recognize that
$E_0 w \in W^{1,q}(\RR^3)$. Moreover, both $E_0 w$ and $ \nabla E_0 w = E_0 \vvec$ vanish outside
of $\Omega$ and therefore by Theorem 5.29 of \cite{Ada03}, $ E_0 w \big|_{\Omega} = w \in W^{1,q}_0(\Omega) $.
This implies that there exists a sequence $w_n \in C^{\infty}_0(\Omega)$ converging to $w$ in the space
$W^{1,q}_0(\Omega)$.  Using identity \eqref{Green_div_smooth}, we can now conclude that for any $\uvec \in W^p(\text{div};\Omega)$
\begin{align*}
 \ell(\uvec) = &  ( E \uvec ,  E_0 \vvec )_{\RR^3}  +  ( E ( \Div \uvec ) , E_0 w )_{\RR^3}  \\
 = & \lim_{n \to \infty} ( E \uvec , \nabla w_n )_{\RR^3}  + ( E (\Div \uvec ), w_n )_{\RR^3} \\
 = & \lim_{n \to \infty}  ( \uvec , \nabla w_n )_{\Omega} + ( \Div \uvec , w_n )_{\Omega} = 0
\end{align*}
where the second to last step required that $w_n \in C^{\infty}_0(\Omega)$.
\end{proof}

The next step is to study the traces of functions in $\Wdiv$. For all $ \phivec \in C^{\infty}(\overline{\Omega})^3$, we can define 
the normal trace $\gamma_n(\phivec) = \phivec \cdot \normal$ where
we have assumed that the domain is sufficiently smooth to define the unit outwards normal $\normal$. 
The question is to know where this quantity will belong when we extend this to
the weaker space $\Wdiv$.

\begin{theorem} \label{thm:Green_div}
Let $\Omega$ be a bounded $C^1$ domain in $\RR^3$, unit outwards normal $\mathbf{n}$, and $1<p<\infty$. 
Then 
\begin{itemize}
\item[(i)] the mapping $\gamma_n$ defined on $C^{\infty}(\overline{\Omega})^3$ can be extended 
               to a continuous and surjective map $\gamma_n : \Wdivi
                         \longrightarrow  \Besov{1-1/q}{q}{q}' \equiv \Besov{-1/p}{p}{p} $.
\item[(ii)] the following Green's formula holds for all $\uvec \in \Wdivi $ and $v \in \Sobolev{1}{q}$,
\begin{equation}  \label{Green_div}
       ( \uvec, \nabla v) + (\Div \uvec, v) = \langle \gamma_n(\uvec),  \gamma_0( v )   \rangle.
\end{equation}
\end{itemize}
\end{theorem}

\begin{proof} Given the elliptic regularity results in $L^p$, the proof is a natural adaptation of the argument used in \cite{Mon03} for the
proof of their Theorem 3.24.

We begin by examining Green's formula. For any $\phi \in C^{\infty}(\overline{\Omega})$ and 
$\boldsymbol{\psi} \in C^{\infty}(\overline{\Omega})^3$, we have
$$
         (\boldsymbol{\psi} , \nabla \phi) + (\Div \boldsymbol{\psi}, \phi) = \langle  \gamma_n(\boldsymbol{\psi}), \gamma_0( {\phi} ) \rangle .
$$
By density, this must hold for all $\phi \in W^{1,q}(\Omega) $. For $ v \in \Sobolev{1}{q}$,
using the previous identity as a weak definition of $\gamma_n(\boldsymbol{\psi})$, we therefore have
\begin{align*}
   \big| \langle  \gamma_n(\boldsymbol{\psi}), \gamma_0( v )  \rangle \big| 
    \leq & \| \boldsymbol{\psi} \|_{\Lp} \cdot \| \nabla v \|_{L^q(\Omega)} + \| \Div \boldsymbol{\psi} \|_{\Lp} \cdot \| v \|_{L^q(\Omega)} \\
    \leq &  \| \boldsymbol{\psi} \|_{\Wdiv} \cdot \|  v \|_{\Sobolev{1}{q}} , \qquad \forall  \boldsymbol{\psi} \in C^{\infty}(\Omega)^3.
\end{align*}
Pick any $\mu \in \Besov{1-1/q}{q}{q}$ and consider $v \in \Sobolev{1}{q}$ the unique weak solution to
$$
   - \Delta v + v = 0 \quad \text{on $\Omega$} \qquad \gamma_0(v) = \mu \quad \text{on $\partial \Omega$,}
$$
guaranteed by Theorem \ref{thm:elliptic1}. Using the fact that this solution depends continuously on the
boundary conditions,  estimate \eqref{estimate1}, we have the bound
$$
   \big| \langle \gamma_n(\boldsymbol{\psi}),  \mu  \rangle \big|  
          \leq  C  \| \boldsymbol{\psi} \|_{\Wdiv} \| \mu \|_{ \Besov{1-1/q}{q}{q}}  .
$$ 
According to the definition of the dual norm \eqref{dual_norm}, we may conclude that the map $\gamma_n$ can 
be extended by continuity to all of $\Wdiv$
with image in the dual of $ \Besov{1-1/q}{q}{q} $, also known as $\Besov{-1/p}{p}{p}$ 
(since $\partial \Omega$ is closed).
Green's formula \eqref{Green_div} follows immediately by density since $\gamma_0( v ) = \mu  \in \Besov{1-1/q}{q}{q}$.

To prove the surjectivity of $\gamma_n$, choose any $\eta \in \Besov{-1/p}{p}{p}$ and construct 
the unique weak solution $\phi$ to 
\begin{equation} \label{Green_div:eq1}
   - \Delta \phi + \phi = 0 \quad \text{on $\Omega$} \qquad \normal \cdot \nabla \phi = \eta \quad \text{on $\partial \Omega$.}
\end{equation}
By Theorem \ref{thm:elliptic4}, the solution $\phi$ exists and belongs to $\Sobolev{1}{p}$. If we define $\uvec = \nabla \phi \in \Lp^3$
then the weak form of \eqref{Green_div:eq1}  gives the identity
$$
         (\uvec, \nabla \psi) + (\phi, \psi) = 0, \qquad \forall \psi \in C^{\infty}_0(\Omega).
$$
In a distributional sense, we therefore have $\Div \uvec = \phi \in \Lp$ and $\uvec \in \Wdiv$. Moreover,
$\gamma_n(\uvec) = \normal \cdot \uvec = \normal \cdot \nabla \phi = \eta$ and the map is surjective.
\end{proof}

\begin{lemma}  \label{lemma:Wdiv2}
For a bounded $C^1$ domain $\Omega \subset \RR^3$ and $1<p<\infty$, we have
$$
   W^p_0(\text{\emph{div}};\Omega) = \Big\{ \uvec\in  \Wdivi \, \Big| \, \gamma_n(\uvec) = 0 \Big\} .
$$
\end{lemma}

\begin{proof}The proof given here is a non-trivial adaptation of the argument used in \cite{Mon03}, where duality was simpler 
to handle because Hilbert spaces were involved. The trace $\gamma_n$ clearly vanishes on $C^{\infty}_0(\Omega)^3$
and therefore, by density,  $\gamma_n$ vanishes on $W^p_0(\text{div};\Omega)$. This proves the inclusion of 
$W^p_0(\text{div};\Omega)$ into the set on the right.  The other inclusion 
can be demonstrated if we show that when $\ell \in \Wdiv'$ vanishes on $C^{\infty}_0(\Omega)$, then $\ell$ must also vanish
on any $\uvec \in \Wdiv$ for which $\gamma_n(\uvec) = 0$. 

Repeating the trick used in the proof of Theorem \ref{thm:densityWdiv}, we exploit the embedding of
$\Wdiv$ into $L^p(\Omega^{(4)})$ and the induced representation of every $\ell \in \Wdiv'$ as
$$
   \ell(\boldsymbol{\psi}) = (\boldsymbol{\psi}, \vvec) + (\Div \boldsymbol{\psi}, w)
$$
for some $\vvec \in L^q(\Omega)^3$ and $w \in L^q(\Omega)$. For all $\boldsymbol{\phi} \in C^{\infty}_0(\Omega)^3$, 
$\ell(\boldsymbol{\phi}) = 0$ and therefore $\vvec = \nabla w$ in a distributional sense. This implies that $w \in \Sobolev{1}{q}$
and that $\ell$ over $\Wdiv$ can be written as 
\begin{align*}
  \ell(\boldsymbol{\psi}) = & (\boldsymbol{\psi}, \nabla w) + (\Div \boldsymbol{\psi}, w) = \langle \gamma_n(\boldsymbol{\psi}) , \gamma_0( w ) \rangle.
\end{align*}
If $\uvec \in \Wdiv$ is such that $\gamma_n(\uvec) = 0$ then $\ell(\uvec) = 0$ and $\uvec$ belongs to the closure of $C^{\infty}_0(\Omega)^3$.
\end{proof}


We now turn to the set of spaces $\Wcurl$ and $W^p_0(\text{curl};\Omega)$. The analysis is very similar to the
one used for $\Wdiv$  and in fact, a suitable generalization of the differential operators involved here would have allowed us to treat both cases
simultaneously, as Mitrea and Mitrea did in \cite{MitMit02}. The objective is to prove Green's formula over these spaces and, in
particular, to characterize the traces of these spaces. In the case of $\Wcurl$, we will show that the tangential trace operators,
defined for all $ \phivec \in C^{\infty}(\overline{\Omega})^3$ as $\gamma_t(\phivec) = \normal \times \phivec $ and 
$\gamma_T(\phivec) = (\normal \times \phivec) \times \normal$, are well-defined. We begin with a simple extension of
Theorem 5.29 from \cite{Ada03}.

\begin{lemma}  \label{lemma:zero_extension} 
Assume $\Omega$ is a bounded Lipschitz domain, $1<p<\infty$ and $E_0 u  $ is the extension by zero of $u$ to all
of $\RR^3$. Then $u$ belongs to $W^p_0(\text{\emph{curl}};\Omega)$ if and only if $E_0 u  $ belongs to 
$W^p(\text{\emph{curl}};\RR^3)$. 
\end{lemma}

\begin{proof}
Theorem 5.29 in \cite{Ada03} is stated only for $\Sobolev{s}{p}$ but the same proof shows that
this result also holds for $\Wcurl$.
\end{proof}

\begin{theorem} \label{thm:densityWcurl}
Consider a bounded Lipschitz domain $\Omega \subset \RR^3$ and $1 < p <\infty$.
Then the set $C^{\infty}(\overline{\Omega})^3$ is dense in  $\Wcurli$.
\end{theorem}

\begin{proof}
We will show that if $\ell \in \Wcurl'$ vanishes over $C^{\infty}(\overline{\Omega})^3$ then $\ell$ must be
identically zero.

As we did in the proof of Theorem \ref{thm:densityWdiv}, there exists an isometric  
embedding $I: W^p(\text{curl};\Omega) \longrightarrow L^p(\Omega^{(6)})$, $I(\uvec) = (\uvec, \curl \uvec)$.
Given any linear functional over $I(W^p(\text{curl};\Omega)) \subset L^p(\Omega^{(6)})$, the Hahn-Banach Theorem states that
can it be extended to a linear functional over $L^p(\Omega^{(6)})$, that is to an element in $L^q(\Omega^{(6)})$.  The
linear functional can therefore be represented as
$$
   \ell( \uvec) = ( \uvec, \vvec )_{\Omega} + (\curl \uvec, \mathbf{w})_{\Omega}, \quad \forall \uvec \in \Wcurl, 
$$  
for some $\vvec, \mathbf{w} \in L^q(\Omega)^3$. 

Let $E_0$ be the extension by zero operator and let $E$ be any extension operator over $\Lp^3$. Then
$$
      \ell(\uvec) = (E \uvec, E_0 \vvec)_{\RR^3} + (E (\curl \uvec), E_0 \mathbf{w})_{\RR^3}.
$$
For any $\boldsymbol{\psi} \in C^{\infty}_0(\RR^3)^3$, Lemma \ref{lemma:Green_weak} implies that
$$
       (\boldsymbol{\psi}, E_0 \vvec)_{\RR^3} + ( \curl \boldsymbol{\psi}, E_0 \mathbf{w})_{\RR^3} = 0.
$$
In other words, $E_0 \vvec = - \curl E_0 \mathbf{w}$ in a distributional sense. Not only does this imply that
$E_0 \mathbf{w} \in W^q(\text{curl};\RR^3)$, but Lemma \ref{lemma:zero_extension} shows that 
$\mathbf{w} \in W^q_0(\text{curl};\Omega)$. 

We can therefore construct a sequence of $\mathbf{w}_n \in C^{\infty}_0(\Omega)^3$ which converges to $\mathbf{w}$
with respect to the norm in $W^q(\text{div};\Omega)$. For any $\uvec \in W^p(\text{div};\Omega)$, we 
combine these facts with formula \eqref{Green_curl_smooth} to deduce
\begin{align*}
 \ell(\uvec) = &  ( E \uvec , E_0 \vvec)_{\RR^3} +  ( E ( \curl \uvec ) , E_0 \mathbf{w} )_{\RR^3} \\
 = & \lim_{n \to \infty} - ( E \uvec , \curl \mathbf{w}_n)_{\RR^3} +  ( E ( \curl \uvec ),  \mathbf{w}_n )_{\RR^3} \\
 = & \lim_{n \to \infty}  - ( \uvec , \curl \mathbf{w}_n)_{\Omega} +  ( \curl \uvec  , \mathbf{w}_n )_{\Omega} = 0 .
\end{align*}
This shows that the closure of $C^{\infty}(\overline{\Omega})^3$ must be $W^p(\text{curl};\Omega).$
\end{proof}

\begin{theorem}  \label{thm:Green_curl}
Consider a bounded $C^1$ domain $\Omega \subset \RR^3$ and $1<p<\infty$. Then 
\begin{itemize}
\item[(i)] the mapping $\gamma_t$ defined on $C^{\infty}(\overline{\Omega})$ can be extended  
               to a continuous  map $\gamma_t : \Wcurli 
                         \longrightarrow  (\Besov{1-1/q}{q}{q}^3)' \equiv \Besov{-1/p}{p}{p}^3 $.
\item[(ii)] the following Green's formula holds for all $\uvec \in \Wcurli $ and $\vvec \in \Sobolev{1}{q}^3$,
\begin{equation}  \label{Green_curl1}
       ( \curl  \uvec, \vvec) - ( \uvec, \curl \vvec) = \langle \gamma_t(\uvec),  \gamma_0( \vvec ) \rangle .
\end{equation}
\end{itemize}
\end{theorem}

\begin{proof} 
The proof is a straightforward extension of the one used in \cite{Mon03}.
For any  $\phivec, \boldsymbol{\psi} \in C^{\infty}(\overline{\Omega})^3$, we have
\be  \label{Green_curl:eq1}
         ( \curl \phivec, \boldsymbol{\psi})  - ( \phivec, \curl \boldsymbol{\psi}) = 
               \langle \gamma_t( \phivec ), \gamma_0( \boldsymbol{\psi} ) \rangle.
\ee
By the density of $C^{\infty}(\overline{\Omega})^3$ in $W^{1,q}(\Omega)^3$, this formula holds for
all $ \boldsymbol{\psi} \in \Sobolev{1}{q}^3$. For any $\mu \in \Besov{1-1/q}{q}{q}^3$, let
$v \in W^{1,q}(\Omega)$ be the weak solution to $-\Delta v + v =0,$ $\gamma_0(v)  = \mu$, as provided by Theorem \ref{thm:elliptic1}.
Using \eqref{Green_curl:eq1} and estimate \eqref{estimate1}, for all $\boldsymbol{\phi}\in C^{\infty}(\overline{\Omega})^3$ we have 
$$
  \big| \langle \gamma_t(\phivec), \gamma_0( \mu ) \rangle \big|
          \leq \| \phivec \|_{\Wcurl} \cdot \| v \|_{\Sobolev{1}{q}}   
          \leq  C \| \phivec \|_{\Wcurl} \| \mu \|_{\Besov{1-1/q}{q}{q}}   .
$$
According to the definition of the dual norm \eqref{dual_norm}, the previous estimate 
implies that $\gamma_t: \Wcurl  \longrightarrow  (\Besov{1-1/q}{q}{q}^3)' = \Besov{-1/p}{p}{p}^3$ 
can be defined by continuous  extension and that 
$$
   \| \gamma_t (\phivec) \|_{\Besov{-1/p}{p}{p}} \leq C \| \phivec \|_{\Wcurl}.
$$           
Since $C^{\infty}(\overline{\Omega})^3$ is dense in $\Wcurl$, we conclude that the identity
\eqref{Green_curl:eq1} continues to hold for all $\uvec\in \Wcurl$ and $ \boldsymbol{\psi} \in \Sobolev{1}{q}^3$.
\end{proof}

An intrinsic definition of $\gamma_t(\Wcurl)$ exists and is described in Mitrea \cite{Mitrea}. For our purposes, a simpler approach 
will be sufficient. Writing $\Yp = \gamma_t(\Wcurl)$, we define the following norm on the image
\begin{equation}    \label{defn:norm_Yp}
      \| \mathbf{w} \|_{\Yp} = \inf \left\{  \| \uvec \|_{\Wcurl} \, \big|  \, \uvec \in \Wcurl \text{ and } \gamma_t(\uvec) = \mathbf{w} \right\} . 
\end{equation}
Although the previous Green's formula \eqref{Green_curl1} used $\gamma_0$, we will be using the more precise expression $\gamma_T$
since, for smooth functions, only the tangential component of $\vvec$ appears in the inner product and
$\langle \gamma_t(\uvec),\gamma_0(\vvec)\rangle = \langle \gamma_t(\uvec),\gamma_T(\vvec) \rangle$.

\begin{theorem}  \label{thm:Green_curl2}
Consider a bounded $C^1$ domain $\Omega$ and $1<p<\infty$. Then 
$\Yp$ is a normed topological vector space, the map $\gamma_t: \Wcurli \longrightarrow Y^p(\partial \Omega) $ is continuous
and surjective and 
$\gamma_T: W^q(\text{\emph{curl}};\Omega) \longrightarrow Y^p(\partial \Omega)' $ is well-defined and continuous.
Moreover, for all $\uvec \in \Wcurli $, $\vvec \in W^q(\text{\emph{curl}};\Omega)$, we have
\begin{equation}  \label{Green_curl2}
       ( \curl  \uvec, \vvec) - ( \uvec, \curl \vvec) = \langle \gamma_t(\uvec),  \gamma_T( \vvec ) \rangle .
\end{equation}
\end{theorem}

\begin{proof} 
The image $\Yp$ equipped with the norm \eqref{defn:norm_Yp} 
is easily seen to be a normed topological vector space since $\gamma_t$ is linear.
To show that $\gamma_t$ is continuous, it suffices to observe that for any $\uvec \in \Wcurl$
$$
     \| \gamma_t(\uvec) \|_{\Yp} = \inf_{\substack{\vvec \in \Wcurl \\ \gamma_t(\vvec)=\gamma_t(\uvec)}} 
                  \|  \vvec \|_{\Wcurl } \leq \| \uvec \|_{\Wcurl}.
$$

We now prove that $\gamma_T$ is continuous. First of all, note that for any $s \in \Yp$ and any 
$\uvec \in \Wcurl$ such that $s = \gamma_t(\uvec)$, we have
$$
      \langle s,  \gamma_T( \boldsymbol{\phi} ) \rangle =  ( \curl  \uvec, \boldsymbol{\phi} ) - ( \uvec, \curl \boldsymbol{\phi})  , 
               \qquad \forall \boldsymbol{\phi} \in C^{\infty}(\overline{\Omega})^3 .
$$
Note that the right hand side is well-defined for all $\uvec \in \Wcurl$ and $\boldsymbol{\phi} \in W^q(\text{curl};\Omega)$. Moreover,
the value on the right-hand side is independent of the choice of $\uvec \in \Wcurl$ so long as $\gamma_t(\uvec) = s$. 
Therefore, for any fixed $s\in \Yp$,  the linear functional 
$$ 
     L(\boldsymbol{\phi}) = \langle s , \gamma_T(\boldsymbol{\phi}) \rangle
$$  
is well-defined for all $\boldsymbol{\phi} \in W^q(\text{curl};\Omega)$. In fact, 
\begin{align*}
       \big| L(\boldsymbol{\phi}) \big| = & \big|  ( \curl  \uvec, \boldsymbol{\phi} ) - ( \uvec, \curl \boldsymbol{\phi})  \big| \\
             = & \inf_{\substack{\vvec \in \Wcurl \\ \gamma_t(\vvec)=s}} \big|  ( \curl  \vvec, \boldsymbol{\phi} ) - ( \vvec, \curl \boldsymbol{\phi}) \big|  \\
         \leq &  \inf_{\substack{\vvec \in \Wcurl \\ \gamma_t(\vvec)=s}} \| \vvec \|_{\Wcurl} \| \boldsymbol{\phi} \|_{W^q(\text{curl};\Omega)} \\
             = &  \| s  \|_{\Yp}  \| \boldsymbol{\phi} \|_{W^q(\text{curl};\Omega)} ,
\end{align*}
which shows that $L$ is continuous over $W^q(\text{curl};\Omega)$. Considering the definition
of the norm on $\Yp'$, equation \eqref{dual_norm}, we immediately see that the previous
identity also implies that 
$$
     \|  \gamma_T( \boldsymbol{\phi}) \|_{\Yp'} \leq \| \boldsymbol{\phi} \|_{W^q(\text{curl};\Omega)}.
$$
Hence, the map $\gamma_T$ is  continuous with image in $\Yp'$.
\end{proof}

\begin{lemma}  \label{lemma:Wcurl2}
For a bounded $C^1$ domain $\Omega \subset \RR^3$ and $1<p<\infty$, we have
$$
   W^p_0(\text{\emph{curl}};\Omega) = \Big\{ \uvec\in  \Wcurli  \, \Big| \, \gamma_t(\uvec) = 0 \Big\} .
$$
\end{lemma}

\begin{proof}
Since $\gamma_t(\boldsymbol{\phi}) = 0$ for all $\boldsymbol{\phi} \in C^{\infty}_0(\Omega)^3$, it suffices to prove the opposite inclusion. Given that
$W^p_0(\text{curl};\Omega)$ is the intersection of all $\ell \in \Wcurl'$ which vanish on $C^{\infty}_0(\Omega)^3$,
we must therefore show that such an $\ell$ also vanishes on $\uvec \in \Wcurl$ when $\gamma_t(\uvec) = 0.$

As we did in the proof of Lemma \ref{lemma:Wdiv2}, any functional $\ell \in \Wcurl'$ can be written in the form
$$
   \ell(\uvec) = (\gvec,\uvec) + (\hvec, \curl \uvec) , \qquad \forall \uvec \in \Wcurl,
$$
for some $\gvec, \hvec \in L^q(\Omega)^3$. If $\ell$ vanishes for all $\boldsymbol{\psi} \in C^{\infty}_0(\Omega)^3$, then
this representation shows that $\gvec = - \curl \hvec$ in a distributional sense and therefore that $\hvec \in W^q(\text{curl};\Omega)$.
For $\uvec \in \Wcurl$ such that $\gamma_t(\uvec) = 0$, we then apply identity \eqref{Green_curl2} to deduce that
\begin{align*}
  \ell(\uvec) = & - (\curl \hvec, \uvec ) + (\hvec, \curl \uvec) = \langle \gamma_t(\uvec ), \gamma_T( \hvec ) \rangle = 0.
\end{align*}
Hence $\uvec$ belongs to the closure of $C^{\infty}_0(\Omega)^3$.
\end{proof}

We complete this section with a simple result concerning the spaces which we have introduced.

\begin{lemma}  \label{lemma:reflexivity}
For a bounded Lipschitz domain $\Omega \subset \RR^3$ and $1<p<\infty$, the spaces
$\Wdivi$ and $\Wcurli$ are locally convex and reflexive Banach spaces.
\end{lemma}

\begin{proof} Local convexity is a triviality since the spaces are normed. Recalling the isometries
\begin{align*}
                  I_{\text{div}}(\uvec) = & (\uvec, \Div \uvec) \in L^p(\Omega^{(4)}, \\
                  I_{\text{curl}}(\uvec) = & (\uvec, \curl \uvec) \in L^p(\Omega^{(6)},                  
\end{align*}
introduced in the proofs of Theorems \ref{thm:densityWdiv} and \ref{thm:densityWcurl}, then we see that
$\Wdiv$ and $\Wcurl$ is isometric to a closed subspace of a reflexive space and are therefore, by Theorem 1.22 of \cite{Ada03},
also reflexive.
\end{proof}

\section{Divergence and rotational free vector fields}
\label{sec:divergence_free}

The objective of this section is to derive fundamental results concerning the existence of divergence free and rotational free vector fields. 
These results appear to be well-known in the literature but we have either not found the results in the form we 
needed or within the same reference. The fundamental result of this section
is a Friedrich's inequality for problems with homogeneous Dirichlet boundary conditions.  This result
is the only one which appears to be entirely new to the literature. Give more detailed citations for the 
two theorems in this section : Dauge ...

%
%

\begin{theorem}  \label{thm:curl_free}
Let $\Omega$ be a bounded simply connected Lipschitz domain in $\RR^3$. Suppose that $\uvec \in \Lp^3$. Then
$\curl \uvec =0$ in $\Omega$ if and only if there exists  $\phi \in \Sobolev{1}{p}$ such that $\uvec = \nabla \phi$. If this is the case, then 
$\phi$ is unique up to an additive constant and 
$$
        \|  \phi \|_{\Sobolev{1}{p} \setminus \RR}  \leq C \| \uvec \|_{\Lp^3}.
$$
\end{theorem}

\begin{proof} The proof follows the one for Theorem 3.37 from \cite{Mon03} which itself goes back to a proof in \cite{AmrBerDauGir98}.
The only differences are the use of convolutions in $L^p$ and of the $L^p$ version of Lemma \ref{thm:Necas}.

Given that $\uvec \in \Lp^3$, $\uvec$ can be extended by zero to all of $\RR^3$, which we denote by
$\mathbf{\widetilde{u}}$. With the help of a mollifier $\rho_{\epsilon}$,  a suitably smooth approximation  can be constructed
$$
       \mathbf{\widetilde{u}}_{\epsilon} = \rho_{\epsilon} \star \mathbf{\widetilde{u}} \in L^p(\RR^3)^3 \cap C^{\infty}_0(\RR^3)^3 ;
$$ 
see Lemma 3.16 in \cite{Ada03}. Moreover, $\mathbf{\widetilde{u}}_{\epsilon} \to \mathbf{\widetilde{u}}$ as $\epsilon \to 0$ in $L^p(\RR^3)^3$.

Under our assumptions on $\Omega$, it is possible to produce a sequence of simply connected Lipschitz domains $\{ \mathcal{O}_j\}_j$ satisfying
$\mathcal{O}_{j} \subset \mathcal{O}_{j+1}$, $\overline{\mathcal{O}_{j}} \subset \Omega$ and $\Omega = \cup_{j=1}^{\infty} \mathcal{O}_{j}$.
Fixing $\mathcal{O}_{j}$, since $\mathbf{\widetilde{u}}_{\epsilon}$ is smooth and 
$$ 
   \curl \mathbf{\widetilde{u}}_{\epsilon} = \curl \big( \rho_{\epsilon} \star \mathbf{\widetilde{u}} \big) 
        = \rho_{\epsilon} \star \big( \curl \mathbf{\widetilde{u}} \big) = \rho_{\epsilon} \star \big( \curl \uvec \big) = 0,
$$
by integration we can construct $\phi^j_{\epsilon} \in C^{\infty}(\Omega)$ such that 
$$
     \mathbf{\widetilde{u}}_{\epsilon}  = \nabla \phi^j_{\epsilon} ,  \qquad
        \int_{\mathcal{O}_j} \phi^j_{\epsilon} \, d\xvec = 0,
        \qquad \text{and }  \qquad \| \phi^j_{\epsilon} \|_{W^{1,p}(\mathcal{O}_j)} \leq C \| \uvec \|_{L^p(\mathcal{O}_j)^3} .
$$

We now observe that since $\nabla \phi^j_{\epsilon} =  \mathbf{\widetilde{u}}_{\epsilon} \to \mathbf{\widetilde{u}}$ in $L^p(\Omega)^3$, then
 Poincar\'e's inequality implies that 
$\phi^j_{\epsilon} \to \phi^j \in W^{1,p}(\mathcal{O}_j)$. Even in the limit, we have $ \| \phi^j \|_{W^{1,p}(\mathcal{O}_j)} \leq C \| \uvec \|_{L^p(\Omega)^3}$.
 The average of $\phi^j$  over $\mathcal{O}_j$ also vanishes.
At the cost of modifying $\phi^j$ by a constant, we may assume that $\phi^j = \phi^{j-1}$ over $\mathcal{O}_{j-1}$. 
For any compact $K\subset \Omega$,
there exists $k$ such that $K \subset \mathcal{O}_k $ and therefore 
by setting $\phi = \phi_k$ over $\mathcal{O}_k$, we can define  $\phi \in L^p_{\text{loc}}(\Omega)$.
In fact, since $\uvec|_{\mathcal{O}_j} = \nabla \phi^j$ we have that $\nabla \phi = \uvec \in L^p(\Omega)^3 \subset \Sobolev{-1}{p}^3$. 
Using Lemma \ref{thm:Necas}, we conclude that $\phi \in \Lp$. 
\end{proof}

\begin{theorem}  \label{thm:divergence_free}
Consider an exponent $1 < p < \infty$ and a bounded $C^1$ domain $\Omega \subset \RR^3$
such that the boundary $\partial \Omega = \overline{\Gamma_D \cup \Gamma_N}$ with $\Gamma_D$ and $\Gamma_N$ open
and connected subset of $\partial \Omega$.  If   $\uvec \in W^p(\text{\emph{div}};\Omega)$ satisfies
\begin{equation}  \label{compatibility_condition}
    \Div \uvec = 0 \qquad \text{on $\Omega$, and } \qquad \int_{\Gamma_i} \normal \cdot \uvec \, d\xvec = 0, \,\, i=D, N,
\end{equation}
then there exists a vector potential $\mathbf{A} \in W^{1,p}(\Omega)^3$ such that $\uvec = \curl \mathbf{A}$ and $\Div \mathbf{A} = 0.$
Moreover, there exists a constant $C$ for which
$$
      \| \mathbf{A} \|_{\Sobolev{1}{p}^3} \leq C \|  \uvec \|_{\Wdivi} .
$$
\end{theorem}

\begin{proof} 
The proof is identical to the one given for Theorem 3.38 in \cite{Mon03} except for the application of the Marcinkiewicz-Mihlin multiplier theorem;
see Chapter 4, subsection 3.2 of \cite{Ste70}. This ingredient is not new since it was already used by Dauge \cite{Dau88} in a proof of
the same theorem for domains with polyhedral boundaries. Again for completeness, we provide the proof.

The problem will be solved using Fourier transforms and therefore, it will be necessary to construct an extension of $\uvec \in W^p(\text{div};\Omega).$
Our assumptions on the domain imply that $\RR^3 \setminus \Omega$ is formed of two components $\Omega_0$ and $\Omega_1$, one of 
which is unbounded, say $\Omega_0$. 
Consider the following problem over $\Omega_1$ :
$$
   \Delta v_1  = 0 \qquad \text{on $\Omega_1$, and } \qquad \normal \cdot \nabla v_1 = \normal \cdot \uvec \qquad \text{on $\partial \Omega_1$.}
$$
By Theorem \ref{thm:Green_div}, the boundary condition $\normal \cdot \uvec \in B^{-1/p}_{p,p}(\partial \Omega_1)$ possesses sufficient regularity
to guarantee, by Theorem \ref{thm:elliptic3}, a solution $v_1 \in W^{1,p}(\Omega_1)$. Clearly $\nabla v_1$ is divergence free inside $\Omega_1$.

To extend $\uvec$ into $\Omega_0$, choose a large closed ball $K$ whose interior contains the closure of $\Omega$. Over 
$K \setminus \Omega$, consider the problem
\begin{equation*}
\begin{array}{rcll}
    \Delta v_0 & = &  0                                    &  \text{on $K \setminus \Omega$,} \\
    \normal \cdot \nabla v_0 & = & \normal \cdot \uvec &  \text{on $\partial \Omega_0$,} \\
     \normal \cdot \nabla v_0 & = & 0   &  \text{on $\partial K$.} 
\end{array}
\end{equation*}
Again, there exists a unique solution $v_0 \in W^{1,p}(K \setminus \Omega)$. The new vector field $\mathbf{\widetilde{u}}$ 
obtained by patching together $\uvec, \nabla v_0$ and $\nabla v_1 $ according to
$$
\mathbf{\widetilde{u}}(\xvec) = 
\begin{cases}
        \uvec(\xvec) & \text{if $\xvec \in \Omega$,} \\
        \normal \cdot \nabla v_0 (\xvec) & \text{if $\xvec \in \Omega_0$,} \\
         \normal \cdot \nabla v_1 (\xvec) & \text{if $\xvec \in (K \setminus \Omega)$,} \\
         0 & \text{if $\xvec$ is in the complement of $K$,}
\end{cases}
$$
then belongs to $\Wdiv$ since the normal components are equal along the boundaries. The
proof that this new function belongs to $L^p$ requires an obvious extension of Lemma 5.3 in \cite{Mon03}.
Furthermore, this extension of $\uvec$ to all of $\RR^3$ satisfies $\curl \mathbf{\widetilde{u}} = 0$.

With the extension $ \mathbf{\widetilde{u}} $ belonging to $L^p(\Omega) \subset L^1(\Omega) $, the Fourier
transform of $\mathbf{\widetilde{u}}$ is well-defined. As in the $L^2$ case, we want to consider the vector
potential $\mathbf{A}$ defined by
$$
     \widehat{A}_1 = \frac{\xi_3 \widehat{u}_2 - \xi_2 \widehat{u}_3}{| \boldsymbol{\xi} |^2}, \qquad
     \widehat{A}_2 = \frac{\xi_1 \widehat{u}_3 - \xi_3 \widehat{u}_1}{| \boldsymbol{\xi} |^2}, \qquad
     \widehat{A}_3 = \frac{\xi_2 \widehat{u}_1 - \xi_1 \widehat{u}_2}{| \boldsymbol{\xi} |^2},
$$ 
but a priori, it is not clear that the inverse Fourier transform of $\mathbf{A}$ can be taken. In this vein, write 
$$
    \widehat{\mathbf{A}}(\boldsymbol{\xi}) 
         = \chi (\boldsymbol{\xi})  \widehat{\mathbf{A}}(\boldsymbol{\xi}) + \big(1-\chi(\boldsymbol{\xi}) \big)  \widehat{\mathbf{A}}(\boldsymbol{\xi})
$$
where $\chi \in C^{\infty}(\RR^3)$ is identically $1$ outside a large ball centered at the origin while being
identically $0$ in a neighborhood of the origin. The inverse Fourier transform of $(1-\chi)\widehat{\mathbf{A}}$ is analytic
in $\RR^3$ and therefore in $\Lp^3$. On the other hand,  the inverse Fourier transform of $\chi \widehat{\mathbf{A}}$
and $\xi_j \chi(\boldsymbol{\xi}) \widehat{\mathbf{A}}_i$ can be written as
\begin{gather*}
       \big( \chi(\boldsymbol{\xi}) \widehat{\mathbf{A}}(\boldsymbol{\xi}) \big)\check{} 
                =  \big( W(\boldsymbol{\xi}) \widehat{\uvec}(\boldsymbol{\xi}) \big)\check{} \\
         \big( \xi_j \chi(\boldsymbol{\xi}) \widehat{A}_i (\boldsymbol{\xi}) \big)\check{} 
                =  \big( \xi_j W(\boldsymbol{\xi}) \widehat{\uvec}(\boldsymbol{\xi}) \big)\check{} 
\end{gather*}
where the matrix $W(\boldsymbol{\xi})$ possesses terms that are asymptotically of the form $\xi_k/|\boldsymbol{\xi}|^2$.
A simple computation shows that the multipliers $W(\boldsymbol{\xi})$ and $\xi_jW(\boldsymbol{\xi})$ satisfies the conditions of the Marcinkiewicz-Mihlin 
multiplier theorem and therefore, that there exists a constant $C$ such that
$$
     \|  \big( \chi(\boldsymbol{\xi}) \widehat{\mathbf{A}}(\boldsymbol{\xi}) \big)\check{}  \|_{\Lp^3} \leq C \| \uvec \|_{L^p(\Omega)^3}
     \quad \text{and} \quad  \|  \big( \xi_j \chi(\boldsymbol{\xi}) \widehat{A}_i (\boldsymbol{\xi}) \big) \check{}  \|_{\Lp} \leq C \| \uvec \|_{L^p(\Omega)^3} .
$$  
This shows that $\mathbf{A}$ in fact belongs to $\Sobolev{1}{p}^3$. 
Using the fact that $\uvec$ is divergence free, it is easy to verify that $\mathbf{A}$ is also divergence free
and that $\uvec = \curl \mathbf{A}$. 
\end{proof}

\section{Friedrich's inequality}
\label{sec:Friedrichs}

The objective of this section is to provide
detailed estimates concerning the regularity of the spaces described in the preceding sections, in particular 
with regard to the boundary conditions. We will show that the boundary condition $\normal \times \Bvec = 0 \in L^p(\partial \Omega)$
provides some additional smoothness inside the domain.
The most important result will be the Friedrich's inequality, Theorem \ref{Friedrich_inequality1}.

We begin our analysis with the following space
\begin{equation}
    \Wimp = \left\{  \uvec \in \Wcurl \,  \Big| \, \normal \times  \uvec  \in L^p_t(\partial \Omega)\right\},
\end{equation}
where $L^p_t(\partial \Omega) = \{ \uvec \in \Lp^3 | \, \normal \cdot \uvec = 0 \}$. 
The norm over $ \Wimp $ will be
$$
           \| \uvec \|_{\Wimp} := \| \uvec \|_{\Lp^3} + \| \curl \uvec \|_{\Lp^3} + \| \normal \times  \uvec  \|_{L^p(\partial \Omega)^3}.
$$ 
The next result shows that this strengthened requirement at the boundary produces a non-trivial amount of regularity 
inside the domain. The result comes from Costabel \cite{Cos90}
where it was presented in $L^2$. We also add that this result bears a close resemblance to a result  mentioned
in Remark 2.19 of \cite{AmrBerDauGir98} and attributed to Dauge \cite{Dau92}. Although the next bound is interesting in it's
own right, its true purpose is to demonstrate Corollary \ref{cor:compactness1}.

\begin{theorem}  \label{thm:lifting_to_3/2}
Consider a bounded $C^1$ domain $\Omega \subset \RR^3$ and $ 1< p < \infty$.
If $\uvec \in \Wdivi \cap \Wcurli$ and $\normal \times \uvec  \in \Lp^3$, then
$\uvec \in \Sobolev{1/p}{p}^3$ and for some constant $C$,
\begin{equation}     \label{estimate_lifting}
        \|  \uvec \|_{\Sobolev{1/p}{p}^3} \leq C \left\{  \| \uvec\|_{\Lp^3} + \| \Div \uvec \|_{\Lp} 
                              + \| \curl \uvec \|_{\Lp^3} + \| \uvec \times \normal \|_{L^p(\partial \Omega)^3} \right\} .
\end{equation}
\end{theorem}

\begin{proof}
The proof is identical to the one given by Monk  \cite{Mon03} for his Theorem 3.47 except that his elliptic regularity result, Theorem 3.17, must 
be replaced by the generalization, Theorem \ref{thm:elliptic_regularity}. 

To prove estimate \eqref{estimate_lifting}, it suffices to prove a local version of this estimate and then to use a partition of unity to extend the estimate
to all of $\Omega$. Following \cite{Mon03}, the regularity of the boundary allows us to assume that locally $\Omega$ is a bounded simply connected
domain with a connected $C^1$ boundary. Since $\uvec \in \Wcurl$ and $\Div \curl \uvec = 0$, then  
there exists a vector potential $\mathbf{w} \in \Sobolev{1}{p}^3$ such that $\curl \mathbf{w} = \curl \uvec$ and  $\Div \mathbf{w} = 0$
(Note that the constraint on the boundary is satisfied because of formula \eqref{Green_curl2}). Secondly, for 
$\mathbf{z} = \uvec - \mathbf{w} \in \Wdiv$, we have $\curl \mathbf{z} = 0$ and by Theorem \ref{thm:curl_free} there exists $\xi \in \Sobolev{1}{p}$ such
that $\mathbf{z} = \nabla \xi$. Finally, Theorem 5.4 of \cite{Tay81} states that we can construct $\eta \in \Sobolev{2}{p}$ satisfying
$\Delta \eta = \Div \mathbf{z} \in \Lp$ (in constrast to the previous $\xi$, to obtain $\eta$ we need to extend $\Div \mathbf{z}$ by zero outside of $\Omega$ 
and thereby use smoother boundary conditions). 

To obtain the increased reguarity, we will find some additional
regularity at the boundary and invoke Theorem \ref{thm:elliptic_regularity} to lift the regularity to all of $\Omega$. 
Since $\normal \times \uvec \in L^p_t(\partial \Omega)^3$ and $\mathbf{w} \in \Sobolev{1}{p}^3$, then
$$
     \normal \times  \mathbf{z}  = \normal \times ( \uvec - \mathbf{w})  \in L^p_t(\partial \Omega) .
$$
Using the fact that $\xi \in \Sobolev{1}{p}$ and $\eta \in \Sobolev{2}{p}$ we have that the function $\nu = \xi - \eta$
satisfies $\normal \times \mathbf{z}  = \normal \times \nabla \xi = \normal \times \nabla (\eta + \nu ) $ 
and $\nabla \nu \times \normal \in L^p_t(\partial \Omega)$. Therefore $\nu \in W^{1,p}(\partial \Omega)$, 
$\Delta \nu = \Div ( \nabla \xi - \nabla \eta) = 0$ 
and Theorem \ref{thm:elliptic_regularity} implies that $\nu \in \Sobolev{1+1/p}{p}$. 

We begin with an estimate for $\nu$. First of all, notice that since $\xi$ is a consequence of Theorem \ref{thm:curl_free}, it is therefore 
defined up a to constant. The function $\nu$ is thus also defined up to a constant and we may assume that it's average value along the boundary 
vanishes.  Using the elliptic regularity result, Theorem \ref{thm:elliptic_regularity}, and Poincar\'e's inequality on the boundary, we deduce the following.
\begin{align}
   \| \nabla \nu \|_{\Sobolev{1/p}{p}} 
                  \leq &  \|   \nu  \|_{\Sobolev{1+1/p}{p}^3}     \nonumber   \\
                  \leq & C \left\{  \| \nu \|_{L^p(\partial \Omega)} + \| \normal \times \nabla \nu \|_{L^p(\partial \Omega)^3} \right\}    \nonumber  \\
                  \leq & C \left\{  \| \normal \times \nabla \nu \|_{L^p(\partial \Omega)^3} + \left|  \int_{\partial \Omega} \nu \, d\boldsymbol{\sigma} \right| \right\} 
                   = C \| \normal \times \nabla \nu \|_{L^p(\partial \Omega)^3}   \nonumber
\end{align} 
Using the identity $\uvec = \mathbf{w} + \nabla \eta + \nabla \nu$ and the stability estimates of
Theorems \ref{thm:Green_curl2}, \ref{thm:divergence_free} and \ref{thm:curl_free} for $\mathbf{w}$ and $\eta$, we further obtain
\begin{align}
   \| \nabla \nu \|_{\Sobolev{1/p}{p}^3} 
                  \leq & C \left\{      \| \normal \times \uvec            \|_{L^p(\partial \Omega)^3} 
                                            +  \| \normal \times \mathbf{w}  \|_{L^p(\partial \Omega)^3} 
                                            +  \| \normal \times \nabla \eta  \|_{L^p(\partial \Omega)^3}  \right\}    \nonumber  \\
                  \leq & C \left\{      \| \normal \times \uvec            \|_{L^p(\partial \Omega)^3} 
                                            +  \| \mathbf{w}  \|_{\Sobolev{1}{p}^3} 
                                            +  \|  \nabla \eta  \|_{\Wcurl}  \right\}    \nonumber  \\
                  \leq & C \left\{      \| \normal \times \uvec            \|_{L^p(\partial \Omega)^3} 
                                            +  \| \curl \uvec  \|_{\Lp^3} 
                                            +  \| \Div \uvec \|_{\Lp}  \right\} . \nonumber
\end{align} 
It is already known that $\uvec \in \Sobolev{1/p}{p}^3$ and we can estimate its norm 
by exploiting the previous estimate. 
\begin{align*}
       \|  \uvec \|_{\Sobolev{1/p}{p}^3} 
                     \leq & C \left\{      \| \mathbf{w} \|_{\Sobolev{1/p}{p}^3} 
                                               + \|  \nabla \eta \|_{\Sobolev{1/p}{p}^3} 
                                               + \|  \nabla \nu  \|_{\Sobolev{1/p}{p}^3} \right\} \\
                     \leq & C \left\{      \| \curl \uvec  \|_{\Lp^3} 
                                               + \| \Div \uvec   \|_{\Lp} 
                                               + \| \normal \times \uvec \|_{L^p(\partial \Omega)^3} \right\}                                                
\end{align*}
This completes the proof of the theorem.
\end{proof}

We introduce the spaces
\begin{align}
     X^p_N  = & \left\{  \uvec \in \Wdiv \cap \Wcurl \, \big| \, \normal \times \uvec = 0 \text{ on $\partial \Omega$}  \right\} , \\
     X^p_{N,0}  = & \left\{  \uvec \in X^p_N \, \big| \, \Div \uvec = 0 \text{ in $ \Omega$}  \right\} , \\
     W^p_N = & \left\{  \uvec \in \Wdiv \cap \Wcurl \, \big| \, \normal \times \uvec \in L^p_t(\partial \Omega)
            \text{ and $\Div \uvec = 0$}  \right\} , \\
     X^p_T  = & \left\{  \uvec \in \Wdiv \cap \Wcurl \, \big| \, \normal \cdot \uvec = 0 \text{ on $\partial \Omega$}  \right\} , \\
     X^p_{T,0}  = & \left\{  \uvec \in X^p_T \, \big| \, \Div \uvec = 0 \text{ in $ \Omega$}  \right\} , \\
     W^p_T = & \left\{  \uvec \in \Wdiv \cap \Wcurl \, \big| \, \normal \cdot \uvec \in L^p(\partial \Omega)
            \text{ and $\Div \uvec = 0$}  \right\} .
\end{align}
The previous theorem implies that these spaces imbed continuously into $\Sobolev{1/p}{p}^3$. 
Theorem 1.4.3.2 in \cite{Gri85} states that $\Sobolev{1/p}{p}$ imbeds compactly into $\Lp$ and therfore, we have
the following important corollary which generalizes Corollary 3.49 of Monk \cite{Mon03}.

\begin{corollary}  \label{cor:compactness1}
Consider a bounded $C^1$ domain  $\Omega \subset \RR^3$ and $1 < p < \infty$. 
Then the spaces $X^p_N, X^p_{N,0}, W^p_N, X^p_T, X^p_{T,0}$ and $W^p_T$ imbed compactly into $\Lp^3.$
\end{corollary}

\begin{theorem}[Friedrich's inequality]         \label{Friedrich_inequality1}
Consider a bounded $C^1$ domain $\Omega \subset \RR^3$ with a connected boundary. Assume the 
exponent satisfies $1 < p < \infty$. Then there exists a constant $C$ such that for all $\uvec \in W_N^p$
$$
   \| \uvec \|_{\Lp^3} \leq C \left\{  \| \curl \uvec \|_{\Lp^3} + \| \normal \times \uvec \|_{L^p_t(\partial \Omega)} \right\}.
$$
Similarly, there exists a constant $C$ such that for all $\uvec \in W^p_T$
$$
   \| \uvec \|_{\Lp^3} \leq C \left\{  \| \curl \uvec \|_{\Lp^3} + \| \normal \cdot \uvec \|_{L^p(\partial \Omega)} \right\}.
$$
\end{theorem}

Note: to prove the second estimate, we need the analogue of Theorem \ref{thm:lifting_to_3/2} when $\normal \cdot \uvec \in L^p(\partial \Omega)$.

\begin{proof}
The proof is an extension of the one given for Corollary 3.51 in \cite{Mon03}. 

Proceeding by contradiction with the first inequality. 
Suppose that there exists a sequence $\uvec_n$ of functions in $W^p_N$ such that
$$
       \| \curl \uvec_n \|_{\Lp^3} + \| \normal \times \uvec \|_{L^p(\partial \Omega)^3} \leq 1/n,
$$
while $\| \uvec_n \|_{\Lp^3} = 1$ for all $n$. Since $W^p_N$ imbeds compactly into $\Lp^3$, then we can extract a subsequence (renamed $\uvec_n$)
which converges to some $\uvec \in \Lp^3$. On the other hand, for all $\boldsymbol{\phi} \in C^{\infty}_0(\Omega)^3$ the curl
of $\uvec$ as a distribution satisfies
\begin{align*}
    \big|  L_{\curl \uvec} (\boldsymbol{\phi}) \big| =  \big| ( \uvec, \curl \boldsymbol{\phi} ) \big| 
        =   \lim_{n\to \infty} \big| (\uvec_n, \curl \boldsymbol{\phi} ) \big| 
        =  \lim_{n \to \infty} \big| ( \curl \uvec_n, \boldsymbol{\phi}) \big|  = 0.
 \end{align*}
In other words, $\uvec \in \Lp^3$ is such that $\curl \uvec = 0 $.

According to Theorem \ref{thm:curl_free}, there exists $\psi \in \Sobolev{1}{p}$ such that
$\uvec = \nabla \psi$. Since $\Div \uvec = 0,$ then $\Delta \psi = 0$ and along the boundary
$$
        \nabla_{\partial \Omega} \psi := \normal \times \nabla \psi = \normal \times \uvec = 0.
$$
In conclusion, $\psi$ is constant along the connected boundary and, by uniqueness of solutions to Poisson's problem, $\psi = 0$ and $\uvec=0$.
This contradicts our earlier hypothesis concerning the first inequality. A similar argument could be used to demonstrate the estimate over $W^p_T$.
\end{proof}

Rather than attempt to immediately demonstrate that $C^{\infty}(\overline{\Omega})^3$ is dense in \linebreak 
$\Wimp$, we will first show
that $C^{\infty}(\overline{\Omega})^3$ is dense in the auxiliary space
$$
      \widetilde{W}^p_{\text{{imp}}}(\text{{curl}};\Omega) = 
          \left\{  \uvec \in \Wimp \,\, \big| \,\, \normal \cdot ( \curl \uvec ) = 0 \text{ on $\partial \Omega$} \right\}.
$$

\begin{lemma} \label{lem:auxiliary}
The functions in the subset $ \widetilde{W}^p_{\text{\emph{imp}}}(\text{\emph{curl}};\Omega) $ of $ W^p_{\text{\emph{imp}}}(\text{\emph{curl}};\Omega) $
can be approximated to arbitrary accuracy by functions in $C^{\infty}(\overline{\Omega})^3$.
\end{lemma}

We now show how this lemma, whose proof we defer to the end of this section, can be used to prove the next theorem.

\begin{theorem}
The set $C^{\infty}(\overline{\Omega})^3$ is dense in $ W^p_{\text{\emph{imp}}}(\text{\emph{curl}};\Omega) $.
\end{theorem}
\begin{proof}
The proof given here is a straightforward adaptation of Theorem 3.54 in \cite{Mon03}. As in the proof of Theorem \ref{thm:lifting_to_3/2},
it suffices to demonstrate the density over subsets that are simply connected with connected boundaries.

Choose $\uvec \in \Wimp$ and consider the variational problem : find $\phi \in \Sobolev{1}{p} / \RR$ such that
$$
     \int_{\Omega} \nabla \phi \cdot \nabla \psi \, d\xvec = \int_{\Omega} \curl \uvec \cdot \nabla \psi \, d\xvec, \quad \forall \psi \in \Sobolev{1}{q} /  \RR.
$$ 
It is an exercise to show that the bilinear form $B: \Sobolev{1}{p} / \RR \times \Sobolev{1}{q} / \RR  \to \RR$
$$
      B(\phi, \psi) = \int_{\Omega} \nabla \phi \cdot \nabla \psi \, d\xvec
$$
is continuous, non-degenerate and satisfies an $\inf$-$\sup$ condition of the form \eqref{inf-sup}.
By the Babu\v{s}ka-Lax-Milgram Theorem, Theorem \ref{thm:Lax-Milgram}, this problem possesses a unique solution.

Evaluating the divergence of $\nabla \phi$ in weak form, we find that
$$
      - ( \nabla \phi, \nabla \xi ) = - (\curl \uvec, \nabla \xi) = 0 , \quad \forall \xi \in C^{\infty}_0(\Omega), 
$$
or in other words $\Div \nabla \phi = 0$.
By adjusting the constant term of $\phi$, we can satisfy condition \eqref{compatibility_condition} of Theorem \ref{thm:divergence_free} 
and deduce the existence of $\mathbf{A} \in \Sobolev{1}{p}^3$
such that $\curl \mathbf{A} = \nabla \phi$, $\Div \mathbf{A} = 0$.  Since $C^{\infty}(\overline{\Omega})^3$ is
dense in $\Sobolev{1}{p}^3$, the same smooth functions which approximate $\mathbf{A}$ can be used in $\Wimp$.

Finally, we will show that the second term in the decomposition $\uvec = \mathbf{A} + (\uvec - \mathbf{A})$
belongs to $\widetilde{W}^p_{\text{imp}}(\text{curl};\Omega)$ and therefore, by Lemma \ref{lem:auxiliary}, can also be approximated by functions 
in $C^{\infty}(\overline{\Omega})^3$. For any $\psi \in \Sobolev{1}{q} $, Green's formula \eqref{Green_curl1} can be used to show
\begin{align*}
     \int_{\partial \Omega}  \normal \cdot \curl ( \uvec - \mathbf{A}) \psi \, d\xvec = & 
           \int_{\Omega} \curl (\uvec - \mathbf{a}) \cdot \nabla \psi \, d\xvec = \int_{\Omega} (\curl \uvec - \nabla \phi ) \cdot \nabla \psi \, d\xvec  = 0,
\end{align*}
and therefore $\normal \cdot \curl (\uvec - \mathbf{A}) = 0$ on $\partial \Omega$ and $\uvec - \mathbf{A} \in \widetilde{W}^p_{\text{imp}}(\text{curl};\Omega)$.
\end{proof}

We now return to the proof of Lemma \ref{lem:auxiliary} whose proof is more involved in the $L^p$ setting
than the corresponding Lemma 3.53 in \cite{Mon03}. We therefore propose to begin with the following important
preliminary result.

\begin{lemma}   \label{Bilinear}
Consider a bounded simply connected domain $\Omega \subset \RR^3$ with a connected $C^1$ boundary. Assume the exponent
satisfies $1 < p < \infty$. Then the bilinear form $ B : X^p_{T,0} \times X^q_{T,0} \to \RR$ defined by
$$
       B(\uvec, \vvec ) := \int_{\Omega} \curl \uvec \cdot \curl \vvec \, d\xvec,
$$
is continuous, non-degenerate and satisfies the $\inf$-$\sup$ condition \eqref{inf-sup}. 
\end{lemma}

\begin{proof}
The bilinear form is easily seen to be continuous since, by  Theorem \eqref{Friedrich_inequality1}, 
$\| \curl \uvec \|_{\Lp^3}$ and $\| \curl \vvec \|_{L^q(\Omega)^3}$ are
norms over respectively $X^p_{T,0}$ and $X^q_{T,0}$.

We will only demonstrate the $\inf$-$\sup$ condition because the  "symmetry" of $B$ will imply 
the non-degeneracy. We will proceed by contradiction and therefore, we assume that there exists a sequence $\mathbf{A}_n \in X^p_{T,0}$
for which
\begin{equation}   \label{Bilinear:eq1}
      \sup_{{\phivec \in X^q_{T,0} : \| \phivec \| = 1}} \big|  ( \curl \mathbf{A}_n, \curl \phivec ) \big| < 1/n ,
\end{equation}
while $\| \curl \mathbf{A}_n \|_{\Lp^3}=1$.

Pick any $\boldsymbol{\psi} \in C^{\infty}_0(\Omega)^3$ and construct 
$\xi \in C^{\infty}_0(\overline{\Omega})$ satisfying 
$$
            \Delta \xi = \Div \boldsymbol{\psi} \text{ on $\Omega$, and } \xi = 0 \text{ on $\partial \Omega$.} 
$$
Now let $\boldsymbol{\chi}$ be the function in $C^{\infty}(\overline{\Omega})^3$ given by
$\curl \boldsymbol{\chi} = \boldsymbol{\psi} - \nabla \xi$ and $\Div \boldsymbol{\chi} = 0$, as guaranteed by Theorem \ref{thm:divergence_free}.
Returning to property \eqref{Bilinear:eq1}, we have that for any $\boldsymbol{\psi} \in C^{\infty}_0(\Omega)^3$,
\begin{align*}
            \big| (\curl \mathbf{A}_n , \boldsymbol{\psi} )_{\Omega} \big| = & \big|  (\curl \mathbf{A}_n ,\curl \boldsymbol{\chi} + \nabla \xi )_{\Omega} \big|  \\
            = & \big|  (\curl \mathbf{A}_n ,\curl \boldsymbol{\chi} )_{\Omega} - ( \Div \curl \mathbf{A}_n,  \xi )_{\Omega}   
                +   ( \normal \cdot \curl \mathbf{A}_n, \xi)_{\partial \Omega}   \big|  \\
             = &     \big|  (\curl \mathbf{A}_n ,\curl \boldsymbol{\chi} )_{\Omega} \big| .  
\end{align*}
By the density result of $C^{\infty}_0(\Omega)^3$ in $L^q(\Omega)^3$ and Lemma 2.7 of \cite{Ada03}, this implies that
$\|  \curl \mathbf{A}_n \|_{\Lp^3} < 1/n$. This contradicts our earlier assumption that $\|  \curl \mathbf{A}_n \|_{\Lp^3} = 1$ and proves the result.
\end{proof}

We complete this section with a proof of Lemma \ref{lem:auxiliary}.

\begin{proof}[Lemma \ref{lem:auxiliary}]
As we mentionned earlier in Theorem \ref{thm:lifting_to_3/2}, the proof of a density result can be demonstrated locally and we may assume
that $\Omega$ is a bounded simply connected domain with a connected $C^1$ boundary. Choose $\uvec \in \widetilde{W}^p_{\text{imp}}(\text{curl};\Omega)$.
Lemma \ref{Bilinear} states that $B$ satisfies the conditions of the Babu\v{s}ka-Lax-Milgram theorem and therefore, that there exists a unique solution 
$\mathbf{A} \in X^p_{T,0}$ to 
\begin{equation}   \label{auxiliary:eq1}
              \int_{\Omega} \curl \mathbf{A} \cdot \curl \boldsymbol{\phi} \, d\xvec 
                     =  \int_{\Omega}  \uvec \cdot \curl \boldsymbol{\phi} - \curl \uvec \cdot \boldsymbol{\phi} \, d\xvec, \quad \forall \boldsymbol{\phi} \in W^q_{T,0}.
\end{equation}
For any $\boldsymbol{\psi} \in X^q_{T}$, Green's formula shows that
$$
      \int_{\Omega} \Div \boldsymbol{\psi} \, d\xvec = \int_{\partial \Omega}  \normal \cdot \boldsymbol{\psi} \, d\boldsymbol{\sigma} = 0.
$$
Hence, Theorem \ref{thm:curl_free} states that we can construct $\xi \in \Sobolev{1}{q}$ satisfying
$$
       \Delta \xi = \nabla \cdot \boldsymbol{\psi} \quad \text{on $\Omega$, and } \quad \normal \cdot \nabla \xi = 0 \quad \text{on $\partial \Omega$.}
$$
Then any $\boldsymbol{\psi} \in X^q_T$ can be decomposed as $ \nabla \xi + \big( \boldsymbol{\psi} - \nabla \xi \big)$ where the two
 two components are such that
$\nabla \xi \in X^q_{T}$ and $\boldsymbol{\psi} - \nabla \xi \in X^q_{T,0}$.  In fact, if $\mathbf{A}$ satisfies equation \eqref{auxiliary:eq1}, 
then exploiting the previous decomposition  we find that for all $\boldsymbol{\psi} \in X^q_T$,
\begin{align}
     \int_{\Omega}  \curl \mathbf{A}  \cdot \curl \boldsymbol{\psi} \, d\xvec = & 
          \int_{\Omega} \curl \mathbf{A} \cdot \curl \big( \boldsymbol{\psi} - \nabla \xi \big) \, d\xvec         \nonumber  \\
         = & \int_{\Omega}  \uvec \cdot \curl\big( \boldsymbol{\psi} - \nabla \xi \big)  - \curl \uvec \cdot \big( \boldsymbol{\psi} - \nabla \xi \big) \, d\xvec     \nonumber  \\ 
         = & \int_{\Omega}  \uvec \cdot \curl  \boldsymbol{\psi}   - \curl \uvec \cdot \boldsymbol{\psi} \, d\xvec     \nonumber  \\
             &    + \int_{\Omega} \uvec \cdot \curl \big( -\nabla \xi \big) \, d\xvec 
                + \int_{\partial \Omega}  \normal \times \uvec \cdot ( -\nabla \xi ) \, d\boldsymbol{\sigma}       \nonumber   \\
         = &  \int_{\Omega}  \uvec \cdot \curl   \boldsymbol{\psi}   - \curl \uvec \cdot  \boldsymbol{\psi} \, d\xvec  .       \label{auxiliary:eq2}
\end{align}
By Theorems \ref{thm:densityWdiv} and   \ref{thm:densityWcurl}, constraint \eqref{auxiliary:eq1} must also be satisfied by all
$\boldsymbol{\psi} \in C^{\infty}_0(\Omega)^3$. Furthermore, Green's formula \eqref{Green_curl2} applied to the right-hand side of \eqref{auxiliary:eq1} provides
the identity
$$
      \int_{\Omega} \curl \mathbf{A} \cdot \curl \boldsymbol{\psi} \, d\xvec = \langle  \gamma_t(\uvec) , \gamma_T( \boldsymbol{\psi}) \rangle_{\partial \Omega} = 0,
$$ 
and therefore, in a distributional sense, $\curl (\curl \mathbf{A}) = 0$.

For the original $\uvec \in \widetilde{W}^p_{\text{imp}}(\text{curl};\Omega)$, we write $\uvec = \curl \mathbf{A} + (\uvec - \curl \mathbf{A})$. Applying
Green's formula \eqref{Green_curl1} to both sides of  \eqref{auxiliary:eq2}, we find that
$$
        \langle \gamma_t( \curl \mathbf{A} ) , \gamma_T( \boldsymbol{\psi} )  \rangle_{\partial \Omega} 
            =    \langle \gamma_t( \uvec ) , \gamma_T(  \boldsymbol{\psi} )  \rangle_{\partial \Omega} ,
$$
that is $\normal \times (\uvec - \curl \mathbf{A}) = 0$. In particular,  $\uvec - \curl \mathbf{A}$ belongs to $ W^p_0(\text{curl};\Omega)$ 
and can be approximated by functions in $C^{\infty}_0(\Omega)^3$.

Continuing with the decomposition $\uvec = \curl \mathbf{A} + (\uvec - \curl \mathbf{A})$ of 
$\uvec \in \widetilde{W}^p_{\text{imp}}(\text{curl};\Omega)$, we now 
show that $\curl \mathbf{A}$ can be approximated by smooth functions. Since $\curl (\curl \mathbf{A}) = 0$, there exists $\xi \in \Sobolev{1}{p}$
for which $\curl \mathbf{A} = \nabla \xi$. The gradient along the boundary is given by
$$
     \nabla_{\partial \Omega} \xi = ( \normal \times \nabla \xi) \times \normal = ( \normal \times \curl \mathbf{A} ) \times \normal
           = ( \normal \times \uvec) \times \normal \in L^p_t(\partial \Omega).
$$
Therefore, $\gamma_0(\xi) \in W^{1,p}(\partial \Omega)$. Recalling Theorem \ref{thm:Crouzeix}, 
we conclude that $\xi$ can be approximated by functions in $C^{\infty}(\overline{\Omega})$. Say that $\xi_n \in C^{\infty}(\overline{\Omega})$
is a sequence converging to $\xi$ in the graph space $G$, previously defined in the statement of Theorem \ref{thm:Crouzeix}. Then the convergence
inside the domain gives us that
\begin{align*}
        \nabla \xi_n & \to \curl \mathbf{A} \text{ in } \Lp^3, \\
        \curl \nabla \xi_n & \to \curl (\curl \mathbf{A}) \text{ in } \Lp^3,
\end{align*}
while the convergence along the boundary guarantees that 
$$
       \normal \times \nabla \xi_n \to \normal \times (\curl \mathbf{A}) \text{ in } L^p_t(\partial \Omega) .
$$
According to the definition of the norm over $\Wimp$, the three previous limits imply that 
$\nabla \xi_n \to \curl \mathbf{A} $ in $\widetilde{W}^p_{\text{imp}}(\text{curl};\Omega)$.
\end{proof}

\begin{theorem}[Helmholtz decomposition]   \label{thm:Helmholtz}
Consider a domain $\Omega$ with a smooth $C^1$ boundary and an exponent $2 \leq p <\infty$. 
Then $W^p_0(\text{\emph{curl}};\Omega)$ has the following direct sum decomposition
$$
    W^p_0(\text{\emph{curl}};\Omega) = X_0 \oplus  \nabla W^{1,p}_0(\Omega), 
$$ 
where 
$$
   X_0 =\left\{  \uvec \in W^p_0(\text{\emph{curl}};\Omega)  \,  \Big|  \,  ( \uvec, \nabla \phi) \text{ for all $\phi \in W^{1,p}_0(\Omega)$}  \right\}.
$$ 
\end{theorem}

\begin{proof}

\end{proof}

\section{Applications to the analysis of a nonlinear elliptic problem}
\label{sec:applications}

With the results given so far, we are in a position to study the existence and uniqueness of solutions to the problem
\eqref{p-curl}-\eqref{Dirichlet_BC}.  In anticipation of our analysis of the associated homogeneous parabolic problem,
we will consider the problem with a source term $\mathbf{S}$ satisfying $\Div \mathbf{S} = 0$ and $\normal \times \mathbf{S} = 0$
along the boundary $\partial \Omega$.
The elliptic $p$-CurlCurl written in  weak form is therefore : given $\mathbf{S} \in X_0$,  
find $\uvec \in X_0$ such that
\be  \label{stationnary_p-curl}
       \left( |\curl \uvec |^{p-2} \curl \uvec, \curl \vvec  \right)_{\Omega} = \left( \mathbf{S} , \vvec \right)_{\Omega}  , 
         \,\, \forall \vvec \in X_0.
\ee
This suggests we consider the nonlinear differential operator $A: W^{p}(\text{{curl}};\Omega)  \to W^{p}(\text{{curl}};\Omega)' $,
defined by the pairing
$$
       \left( A \uvec, \vvec  \right) = \int_{\Omega} |\curl \uvec|^{p-2} \curl \uvec \cdot \curl \vvec \, d\xvec.
$$

The next two vector inequalities are well-known and demonstrated in \cite{BarLiu94}. In fact, a generalization
can also be found in \cite{LiuYan01}.

\begin{lemma}[\cite{GloMar75,BarLiu94}] For all $p \in (1, \infty)$ and $\delta \geq 0$ there exists positive constants 
$a_1(p,n)$ and $a_2(p,n)$ such that for all $\xi, \eta \in \RR^n$,
\be  \label{ineq_1}
   | |\xi|^{p-2} \xi -  |\eta|^{p-2}\eta | \leq a_1 |\xi - \eta|^{1-\delta} \big( |\xi| + |\eta|\big)^{p-2+\delta},
\ee
and
\be  \label{ineq_2}
  |\xi-\eta|^{2+\delta}\big( |\xi|+|\eta|\big)^{p-2-\delta}    \leq  a_2 \big( |\xi|^{p-2}  \xi - |\eta|^{p-2}\eta \big) \cdot (\xi - \eta) .
\ee
\end{lemma}

\begin{lemma}  \label{lemma:monotone}
Assume that $\Omega$ is bounded domain with a $C^1$ boundary and that $2 \leq p < \infty$.
Then, the operator $A$ is 
\begin{itemize}
\item[(i)] well-defined and hemicontinuous over $X_0$ : for some $C_1$
\begin{equation}
       \| A \uvec \|_{X_0'} \leq C_1 \| \uvec \|_{X_0}^{p-1}, \qquad \forall \uvec \in X_0;    \label{A_continuity} 
\end{equation}
\item[(ii)] strictly monotone over $X_0$ : for some $C_2$
\begin{equation}
 \|  \uvec -  \vvec \|_{X_0}^p
                  \leq C_2 \left(  A \uvec -A \vvec,   \uvec - \vvec  \right)_{\Omega},    \qquad \forall \uvec, \vvec \in X_0;        \label{A_monotone}
\end{equation}
\item[(iii)] stable over $X_0$ : for some $C_3$
\begin{equation}
        \| A \uvec - A \vvec \|_{X_0'}    
              \leq  C_3  \| \uvec - \vvec  \|_{X_0} \left(   \|  \uvec \|_{X_0}  +  \| \vvec \|_{X_0}  \right)^{p-2},    \qquad \forall \uvec, \vvec \in X_0; 
                       \label{A_stable}
\end{equation}
\item[(iv)] coercive over $X_0$,
\begin{equation}
     \big| (A \uvec, \uvec) \big| \to \infty , \text{ if $\|\uvec\|_{X_0} \to \infty$.} \label{A_coercive}
\end{equation}
\end{itemize}
In fact, (i) and (iii) also hold over $W^p_0(\text{curl},\Omega)$.
\end{lemma}

\begin{proof}
To prove that $A$ is well-defined and hemicontinuous, it suffices to observe that 
\begin{align*}
  \big| (A \uvec, \vvec )\big| \leq &  \int_{\Omega} |\curl \uvec|^{p-1} |\curl \vvec| \, d\xvec  \\
         \leq & \| |\curl \uvec|^{p-1}\|_{L^q(\Omega)^3}  \cdot \| \curl \vvec \|_{\Lp^3}  \\
           = &   \big( \| \curl \uvec \|_{\Lp^3} \big)^{p-1} \cdot \| \curl \vvec \|_{\Lp^3} \leq  \big( \|\uvec \|_{\Wcurl} \big)^{p-1} \| \vvec \|_{\Wcurl} .
\end{align*} 
According to the definition of the dual norm, \eqref{dual_norm}, this implies \eqref{A_continuity}.
To prove \eqref{A_monotone}, we use \eqref{ineq_2} with $\delta = p-2 \geq 0$ and Friedrich's inequality to write
\begin{align*}
   \big(  A \uvec - A \vvec, \uvec - \vvec  \big) = &  \int_{\Omega} \Big[   |\curl \uvec |^{p-2} \curl \uvec - |\curl \vvec|^{p-2} \curl \vvec  \Big] \cdot \curl (\uvec - \vvec) \, d\xvec \\
   \geq & a_2 \int_{\Omega} \big|   \curl \uvec  - \curl \vvec  \big|^p \, d\xvec  = a_2 \| \uvec - \vvec \|^p_{X_0} .
\end{align*} 
To prove \eqref{A_stable}, we begin with the application of H\"{o}lder's inequality 
\begin{align*}
     \| A \uvec  - A  \vvec  & \|_{\Wcurl'}   
              =  \sup_{\phivec \in \Wcurl } \frac{( A \uvec - A \vvec , \phivec ) }{ \| \phivec \|_{\Wcurl} } \\
              = &  \sup_{ \phivec } \frac{1}{ \| \phivec \| } \Bigg|  \int_{\Omega} \Big[ |\curl \uvec|^{p-2} \curl \uvec - |\curl \vvec|^{p-2} \curl \vvec \Big]   
                          \cdot \curl \phivec \, d\xvec \Bigg| \\
          \leq &  \sup_{ \phivec } \frac{1}{ \|\phivec \| } \Bigg|  \int_{\Omega} \left| |\curl \uvec|^{p-2} \curl \uvec - |\curl \vvec|^{p-2} \curl \vvec \right|^{q}
                  \, d\xvec \Bigg|^{1/q} \cdot \| \curl \phivec \|_{\Lp^3} \\
          \leq &  \Bigg|  \int_{\Omega} \left| |\curl \uvec|^{p-2} \curl \uvec - |\curl \vvec|^{p-2} \curl \vvec \right|^{q}
                  \, d\xvec \Bigg|^{1/q} \\                    
\end{align*}
At this point, we recall that $q = p/(p-1)$,  we exploit \eqref{ineq_1} with $\delta = 0$ and then follow up with  another application of H\"{o}lder's inequality
\begin{align*}
     \| A \uvec - A \vvec \|_{\Wcurl'}   \leq & \, a_1 \left\{  \int_{\Omega} \left| \curl \uvec - \curl \vvec \right|^{q} 
                                                                            \cdot \big( |\curl \uvec| + |\curl \vvec| \big)^{(p-2) q}\right\}^{1/q} \\
                  \leq & \, a_1 \Bigg\{        \bigg\{ \int_{\Omega}     \big| \curl \uvec - \curl \vvec \big|^{ \frac{p}{p-1} \cdot (p-1) } \, d\xvec         \bigg\}^{\frac{1}{p-1}}  \\
                          & \phantom{\, C \Bigg\{}
                                           \times   \bigg\{  \int_{\Omega} \big( |\curl \uvec| + 
                                                                                                         |\curl \vvec| \big)^{ (p-2) \cdot \frac{p}{p-1} \cdot \frac{p-1}{p-2} }   \bigg\}^{\frac{p-2}{p-1}}
                                     \Bigg\}^{\frac{p-1}{p}} \\
                  \leq & \, a_1 \| \uvec  - \vvec \|_{X_0}^{\frac{p}{p-1}\frac{p-1}{p}} \times \| |\curl \uvec| + |\curl \vvec| \|_{X_0}^{p \frac{p-2}{p-1}  \frac{p-1}{p}} \\
                  \leq & \, a_1 \| \uvec - \vvec \|_{\Wcurl} \cdot \Big(  \| \uvec \|_{\Wcurl} + \| \vvec \|_{\Wcurl} \Big)^{p-2}         .           
\end{align*}
Finally, Friedrich's inequality implies \eqref{A_coercive} because of the identity
$$
  \big| (A \uvec, \uvec) \big| = \int_{\Omega} | \curl \uvec |^p \, d\xvec = \| \curl \uvec \|_{\Lp^3}^p.
$$
\end{proof}

\begin{theorem}  \label{thm:stationnary1}
Assume that $\Omega$ is bounded domain with a $C^1$ boundary and that $2 \leq p < \infty$.
For any $\mathbf{S} \in X_0$, the problem
\eqref{p-curl}-\eqref{Dirichlet_BC} has a unique weak solution $\uvec \in X_0$. 
Moreover, the inverse operator $A^{-1}$ is continuous.
\end{theorem}

\begin{proof}
Note that $S \in X_0'$ because $p \geq 2$. Lemma \ref{lemma:monotone} states that  $A$ satisfies all the conditions 
Theorem 12.1 from \cite{Die10}, and therefore $A$ is onto $X_0'$. Strict monotonicity implies that the solution is unique.
The existence of a smooth inverse follows from the remarks following Theorem 12.1 in \cite{Die10}.
\end{proof}

%
%
\bibliographystyle{plain}

\end{document}